\DeclareSymbolFont{AMSb}{U}{msb}{m}{n}
\documentclass[noamsfonts, final]{amsart}
\usepackage[usenames,dvipsnames]{xcolor}
\definecolor{cite}{HTML}{0851A6}
\definecolor{url}{HTML}{0851A6}
\definecolor{link}{HTML}{8F0C00}

\usepackage[alphabetic]{amsrefs}

\usepackage[colorlinks=true, linkcolor=link, citecolor = cite, linktocpage]{hyperref}

\usepackage{fancyhdr, amsmath, amsthm, enumerate}
\usepackage[all, cmtip]{xy}

\usepackage{todonotes}

\usepackage[bitstream-charter]{mathdesign}
\usepackage[T1]{fontenc}

\DeclareMathAlphabet{\eur}{U}{zeus}{m}{n}
\newcommand{\matheur}[1]{\eur{#1}}

\usepackage[letterpaper, textheight=8.56in, textwidth=5.95in, twoside, centering]{geometry}

\theoremstyle{plain}
\newtheorem{prop}[subsubsection]{Proposition}
\newtheorem{lem}[subsubsection]{Lemma}
\newtheorem{cor}[subsubsection]{Corollary}
\newtheorem{thm}[subsubsection]{Theorem}
\newtheorem*{thm*}{Theorem}

\theoremstyle{definition}
\newtheorem{defn}[subsubsection]{Definition}
\newtheorem{notation}[subsubsection]{Notation}

\theoremstyle{remark}
\newtheorem{rmk}[subsubsection]{Remark}

\newcommand{\teq}{\addtocounter{subsubsection}{1}\tag{\thesubsubsection}}

\newcommand{\arrdisplacement}{0.36ex}
\newcommand{\arrdisplacementsp}{0.72ex}

\DeclareMathOperator{\addCoFil}{addCoFil}
\DeclareMathOperator{\addFil}{addFil}
\DeclareMathOperator{\assgr}{ass-gr}
\DeclareMathOperator{\BarO}{Bar}
\DeclareMathOperator{\Bun}{Bun}
\DeclareMathOperator{\car}{char}
\newcommand{\Cat}{\matheur{C}\mathrm{at}}
\DeclareMathOperator{\coChev}{coChev}
\DeclareMathOperator{\coFact}{coFact}
\DeclareMathOperator{\coFactstar}{\coFact^\star}
\DeclareMathOperator{\coFib}{coFib}
\DeclareMathOperator{\coFil}{coFil}
\DeclareMathOperator{\coFree}{coFree}

\DeclareMathOperator{\coLie}{coLie}
\DeclareMathOperator{\coLieshriek}{\coLie^!}
\DeclareMathOperator{\coLiestar}{\coLie^\star}
\DeclareMathOperator*{\colim}{colim}
\DeclareMathOperator{\coOp}{coOp}
\DeclareMathOperator{\coPrim}{coPrim}
\DeclareMathOperator{\cotriv}{cotriv}

\DeclareMathOperator{\Chev}{Chev}
\DeclareMathOperator{\coBarP}{coBar}
\DeclareMathOperator{\ComAlg}{ComAlg}
\DeclareMathOperator{\ComAlgshriek}{\ComAlg^!}
\DeclareMathOperator{\ComAlgstar}{\ComAlg^\star}
\DeclareMathOperator{\ComCoAlg}{ComCoAlg}

\newcommand{\ComCoAlgstar}{\ComCoAlg^\star}
\newcommand{\cont}{\mathrm{cont}}
\newcommand{\decay}{\mathrm{decay}}

\newcommand{\DGCat}{\mathrm{DGCat}}
\newcommand{\DGCatpres}{\DGCat_\pres}
\newcommand{\DGCatprescont}{\DGCat_{\pres, \cont}}
\DeclareMathOperator{\diag}{diag}
\newcommand{\disj}{\mathrm{disj}}

\newcommand{\enh}{\mathrm{enh}}

\newcommand{\etale}{\'etale}
\DeclareMathOperator{\Fact}{Fact}
\DeclareMathOperator{\Factstar}{\Fact^\star}
\DeclareMathOperator{\Fib}{Fib}
\newcommand{\Fil}{\mathrm{Fil}}

\newcommand{\Fpbar}{\lbar{\mathbb{F}}_p}
\newcommand{\Fq}{\mathbb{F}_q}
\newcommand{\Fqbar}{\lbar{\mathbb{F}}_q}
\DeclareMathOperator{\Free}{Free}
\newcommand{\fSet}{\mathrm{fSet}}
\DeclareMathOperator{\Fun}{Fun}

\DeclareMathOperator{\Gr}{Gr}
\newcommand{\gr}{\mathrm{gr}}
\DeclareMathOperator{\Hom}{Hom}
\newcommand{\id}{\mathrm{id}}

\DeclareMathOperator{\indnilp}{ind-nilp}
\DeclareMathOperator{\ins}{ins}

\DeclareMathOperator{\Lie}{Lie}
\DeclareMathOperator{\Lieast}{\Lie^\ast}

\DeclareMathOperator{\Liestar}{\Lie^\star}

\newcommand{\oblv}{\mathrm{oblv}}
\newcommand{\op}{\mathrm{op}}
\DeclareMathOperator{\Op}{Op}

\newcommand{\otimesshriek}{\overset{!}{\otimes}}
\newcommand{\otimesstar}{\otimes^\star}
\DeclareMathOperator{\Palg}{-alg}
\DeclareMathOperator{\Pcoalg}{-coalg}

\newcommand{\Poincare}{Poincar\'e}
\newcommand{\pres}{\mathrm{pres}}
\newcommand{\PreStk}{\mathrm{PreStk}}
\newcommand{\Prim}{\mathrm{Prim}}

\newcommand{\pt}{\mathrm{pt}}

\newcommand{\Ql}{\mathbb{Q}_\ell}
\newcommand{\Qlbar}{\lbar{\mathbb{Q}}_\ell}
\DeclareMathOperator{\Ran}{Ran}
\newcommand{\red}{\mathrm{red}}

\DeclareMathOperator{\Sym}{Sym}
\DeclareMathOperator{\SymMon}{SymMon}
\newcommand{\Sch}{\mathrm{Sch}}
\newcommand{\set}{\mathrm{set}}
\DeclareMathOperator{\Spec}{Spec}
\newcommand{\surjects}{\twoheadrightarrow}
\DeclareMathOperator{\Shv}{Shv}

\newcommand{\Spc}{\mathrm{Spc}}
\newcommand{\stab}{\mathrm{stab}}
\DeclareMathOperator{\Supp}{Supp}
\newcommand{\surj}{\mathrm{surj}}
\DeclareMathOperator{\TakeOut}{TakeOut}
\DeclareMathOperator{\Tot}{Tot}
\DeclareMathOperator{\tr}{tr}
\DeclareMathOperator{\triv}{triv}
\newcommand{\union}{\mathrm{union}}
\newcommand{\Vect}{\mathrm{Vect}}

\newcommand{\lbar}[1]{\overline{#1}}
\newcommand{\ardis}{\ar@<\arrdisplacement>}
\newcommand{\ardissp}{\ar@<\arrdisplacementsp>}
\newcommand{\oversetsupscript}[3]{\overset{#2}{#1}{}^{#3}}
\newcommand{\wtilde}[1]{\widetilde{#1}}

\title[The Atiyah-Bott formula and connectivity in chiral Koszul duality]{The Atiyah-Bott formula and connectivity in chiral Koszul duality}
\author{Quoc P. Ho}
\address{Institute of Science and Technology Austria, Klosterneuburg, Austria}
\email{qho@ist.ac.at}
\date{\today}


\begin{document}
\begin{abstract}
	The $\otimesstar$-monoidal structure on the category of sheaves on the $\Ran$ space is not pro-nilpotent in the sense of~\cite{francis_chiral_2011}. However, under some connectivity assumptions, we prove that Koszul duality induces an equivalence of categories and that this equivalence behaves nicely with respect to Verdier duality on the $\Ran$ space and integrating along the $\Ran$ space, i.e. taking factorization homology. Based on ideas sketched in~\cite{gaitsgory_contractibility_2012}, we show that these results also offer a simpler alternative to one of the two main steps in the proof of the Atiyah-Bott formula given in~\cite{gaitsgory_weils_2014} and~\cite{gaitsgory_atiyah-bott_2015}.
\end{abstract}

\maketitle

\tableofcontents

\section{Introduction}

\subsection{History}
Let $X$ be a smooth and complete algebraic curve, and $G$ a simply-connected semi-simple algebraic group over an algebraically closed field $k$.\footnote{This corresponds to the case of constant group $G\times X$ over $X$. For simplicity's sake, we will restrict ourselves to this case in the introduction.} Then we know that
\[
	C^*(BG, \Lambda) \simeq \Sym_\Lambda V
\]
for some finite dimensional vector space $V$ over $\Lambda$, where $\Lambda$ is $\Ql$ when $k=\Fpbar$ ($\ell \neq p$), and $\Lambda$ is any field of characteristic $0$ when $k$ has characteristic 0.

Let $\Bun_G$ denote the moduli stack of principal $G$-bundles over $X$. In the differential geometric setting, i.e. when $k=\mathbb{C}$, the cohomology ring of $\Bun_G$ was computed by Atiyah and Bott in~\cite{atiyah_yang-mills_1983} using Morse-theoretic methods.

\begin{thm}[Atiyah-Bott] \label{thm:intro:Atiyah-Bott}
We have the following equivalence
\[
	C^*(\Bun_G, \Lambda) = \Sym_\Lambda(C^*(X, V \otimes \omega_X)),
\]
where $\omega_X$ is the dualizing sheaf of $X$.
\end{thm}

In the recent work~\cite{gaitsgory_weils_2014}, Gaitsgory and Lurie gave a purely algebro-geometric proof of the theorem above in the framework of \etale{} cohomology (see also~\cite{gaitsgory_atiyah-bott_2015} for an alternative perspective). In the case where $X$ and $G$ come from objects over $k=\Fq$, the isomorphism in Theorem~\ref{thm:intro:Atiyah-Bott} was proved to be compatible with the Frobenius actions on both sides. The Grothendieck-Lefschetz trace formula for $\Bun_G$ then gives an expression for the number of $k$-points on $\Bun_G$ and hence, confirms the conjecture of Weil that the Tamagawa number of $G$ is 1.

Following ideas suggested in~\cite{gaitsgory_contractibility_2012}, this paper aims to provide an alternative (and simpler) proof of one of the two main steps in the original proofs, as given in~\cite{gaitsgory_weils_2014} and~\cite{gaitsgory_atiyah-bott_2015}. This is possible due to a family of new results regarding connectivity in the theory of chiral Koszul duality proved in this paper which are of independent interest.

\subsection{Prerequisites and guides to the literature}
For the reader's convenience, we include a quick review of the necessary background as well as pointers to the existing literature in \S\ref{sec:prelims}. The readers who are unfamiliar with the language used in the introduction are encouraged to take a quick look at \S\ref{sec:prelims} before returning to the current section.

\subsection{A sketch of Gaitsgory and Lurie's method} We will now provide a sketch of the strategy employed by~\cite{gaitsgory_weils_2014} and~\cite{gaitsgory_atiyah-bott_2015}. In both cases, the proofs utilize the theory of factorization algebras. Broadly speaking, there are two main steps: non-abelian \Poincare{} duality and Verdier duality on the $\Ran$ space.

The readers who are only interested in Koszul duality in the setting of factorization algebras in its own rights can safely skip to~\S\ref{subsec:intro_maingoal}.

\subsubsection{Non-abelian \Poincare{} duality} The first step involves a factorizable sheaf $\matheur{A}$ on $\Ran X$ from $f_! \omega_{\Gr_{\Ran X}}$ where $f$ is the natural map
\[
	f: \Gr_{\Ran X} \to \Ran X,
\]
and $\Gr_{\Ran X}$ is the Beilinson-Drinfeld factorizable affine Grassmannian. The crucial observation is that the natural map
\[
	\Gr_{\Ran X} \to \Bun_G
\]
has homologically contractible fibers, and hence, we get an equivalence
\[
	C^*_c(\Bun_G, \omega_{\Bun_G}) \simeq C^*_c(\Ran X, \matheur{A}). \teq\label{eq:intro_nonab_poincare}
\]

\subsubsection{Verdier duality} The right hand side of~\eqref{eq:intro_nonab_poincare} is, however, not directly computable. If one thinks of factorizable sheaves on $\Ran X$ as $E_2$-algebras, then one reason that makes it hard to compute the factorization homology of $\matheur{A}$ is the fact that it's not necessarily commutative (i.e. not $E_\infty$). $\matheur{A}$, however, also has a commutative co-algebra structure, via the diagonal map\footnote{We are eliding a minor, but technical, point about unital vs. non-unital here.}
\[
	\Gr \to \Gr \times \Gr.
\]
Thus, its Verdier dual $D_{\Ran X}\matheur{A}$ naturally has the structure of a commutative algebra. In fact, it is proved that $D_{\Ran X}\matheur{A}$ is a commutative factorization algebra.

\subsubsection{Computing the Verdier dual} One can prove something even better: $D_{\Ran X}\matheur{A}$ is isomorphic to the commutative factorization algebra $\matheur{B}$ coming from $C^*(BG)$. Namely, the \emph{co-stalk} of $\matheur{B}$ at any closed point $\iota_x: x \hookrightarrow X$ is
\[
	\iota_x^! \matheur{B} \simeq C^*(BG)
\]
and in fact
\[
	\matheur{B}|_X \simeq C^*(BG)\otimes \omega_X.
\]

A natural map from one to the other is given by a certain pairing between $\matheur{A}$ and $\matheur{B}$. Since these are factorizable, showing that this map is an equivalence amounts to showing that its restriction to $X$ is also an equivalence. This is now a purely local problem, and hence, for example, one can reduce it to the case of $\mathbb{P}^1$ to prove it.

\begin{rmk}
Note that in the above, co-stalk, rather than stalk, appears. This is because in~\cite{gaitsgory_weils_2014,gaitsgory_atiyah-bott_2015}, sheaves on (pre-)stacks are set up using the $!$-functors rather than $*$-functors. 
\end{rmk}

\subsubsection{Conclusion} Note from the above that
\[
	\matheur{B}|_X \simeq C^*(BG)\otimes \omega_X \simeq \Sym V \otimes \omega_X
\]
is a free commutative algebra, where $V$ is some explicit chain complex that we can compute. But factorization homology with coefficients in a free commutative factorization algebra is easy to compute. Hence, we conclude
\begin{align*}
	C^*(\Bun_G, \Ql)
	&\simeq C^*_c(\Bun_G, \omega_{\Bun_G})^\vee \\
	&\simeq C^*_c(\Ran X, \matheur{A})^\vee \\
	&\simeq C^*_c(\Ran X, D_{\Ran X}\matheur{A}) \teq \label{eq:intro_verdier_homology} \\
	&\simeq C^*_c(\Ran X, \matheur{B}) \\
	&\simeq \Sym C^*_c(X, V \otimes \omega_X).
\end{align*}

\subsection{What does this paper do?}
\label{subsec:intro_maingoal}
In this paper, we prove that, under some connectivity assumptions, Koszul duality on the category of sheaves on the $\Ran$ space with the $\otimesstar$-monoidal structure induces an equivalence of categories and that this equivalence behaves nicely with respect to Verdier duality on the $\Ran$ space and integrating along the $\Ran$ space, i.e. taking factorization homology. This equivalence is different from those appearing in~\cite{francis_chiral_2011} since the $\otimesstar$-monoidal structure is not pro-nilpotent. On the other hands, our results are quite similar to those of Quillen~\cite{quillen_rational_1969} in the sense that by imposing certain connectivity conditions on the objects involved, we can turn Koszul duality into an equivalence.

Even though the results proved in the paper are of independent interest, our main motivation comes from the ideas sketched in~\cite{gaitsgory_contractibility_2012}. While both~\cite{gaitsgory_weils_2014} and~\cite{gaitsgory_atiyah-bott_2015} follow a similar strategy, the latter develops the theory of Verdier duality on prestacks and applies it to the case of the $\Ran$ space, resulting in a more streamlined and simpler proof of the second step. However, since the $\Ran$ space is a big object,\footnote{In the terminology of~\cite{gaitsgory_atiyah-bott_2015}, it's not finitary.} its technical properties in relation to factorization homology and factorizability are difficult to establish. More precisely, it takes a lot of work to prove the (innocent looking) equivalence~\eqref{eq:intro_verdier_homology} and to a somewhat lesser extent, the fact that $D_{\Ran X} \matheur{A}$ is factorizable. This results in a rather complicated technical heart of~\cite{gaitsgory_atiyah-bott_2015}. The results proved in this paper further simplify the second step of the proof. More precisely, these results could be used to replace all of \S8, \S9, and part of \S12 and \S15 of~\cite{gaitsgory_atiyah-bott_2015}.

Note also that many technical results about Verdier duality are proved only for the case of curves in~\cite{gaitsgory_atiyah-bott_2015}, while results stated here about Koszul duality are for arbitrary dimension. This is in part because~\cite{gaitsgory_atiyah-bott_2015} works with more general sheaves on the $\Ran$ space, whereas we mostly concern ourselves with sheaves of special shapes, i.e. those of the form $\Chev \mathfrak{g}$ or $\coChev\mathfrak{g}$.

\subsection{An outline of our results} \label{subsec:intro_outline_results} We will now state the main results proved in this paper.

\subsubsection{Koszul duality for $\Lie$ and $\ComCoAlg$} Let $X \in \Sch$ be a scheme (see \S\ref{subsubsec:ag_convention} for our convention), and $\ComCoAlgstar(\Ran X)$ and $\Lie^\star(\Ran X)$ denote the categories of co-commutative co-algebra objects and Lie algebra objects in $\Shv(\Ran X)$ with respect to the $\otimesstar$-monoidal structure. The theory of Koszul duality developed in~\cite{francis_chiral_2011} gives a pair of adjoint functors\footnote{\label{fn:indnilp_vs_ord} Strictly speaking, we are using the category $\ComCoAlg^{\indnilp}$ of ind-nilpotent commutative co-algebras. However, we will see easily that, subject to an appropriate connectivity assumption of sheaves on $\Ran X$, this category coincides with the category $\ComCoAlg$.}
\[
	\Chev: \Liestar(\Ran X) \rightleftarrows  \ComCoAlgstar(\Ran X): \Prim[-1] \teq\label{eq:intro_koszul_duality}
\]

Even though the pair of adjoint functors above are not mutually inverses of each other in general, they are when we impose certain connectivity constraints on both sides.

\begin{thm}[Theorem~\ref{thm:Koszul_duality_connectivity_on_Ran}] \label{thm:intro:Koszul_duality_connectivity_on_Ran}
	Suppose $X$ is smooth over $k$. Then we have the following commutative diagram
	\[
	\xymatrix{
		\Liestar(\Ran X)^{\leq c_L} \ar@{=}[rr]^<<<<<<<<<<{\Chev}_<<<<<<<<<<{\Prim[-1]} && \ComCoAlgstar(\Ran X)^{\leq c_{cA}} \\
		\Liestar(X)^{\leq c_L} \ar@{^(->}[u] \ar@{=}[rr]^{\Chev}_{\Prim[-1]} && \coFactstar(X)^{\leq c_{cA}} \ar@{^(->}[u]
	}
	\]
	where $\leq c_L$ and $\leq c_{cA}$ denote the connectivity constraints given in Definition~\ref{defn:connectivity_constraints_Lie_ComCoAlg_Ran}, and where $\Chev$ and $\Prim[-1]$ are the functors coming from Koszul duality.
\end{thm}

\subsubsection{Koszul duality for $\coLie$ and $\ComAlg$}  Let $\ComAlgstar(\Ran X)$ and $\coLiestar(\Ran X)$ denote the categories of commutative algebra objects and co-Lie co-algebra objects in $\Shv(\Ran X)$ with respect to the $\otimesstar$-monoidal structure. As above, we have the following pair of adjoint functors\footnote{See also footnote~\ref{fn:indnilp_vs_ord}.}
\[
	\coPrim[1]: \ComAlgstar(\Ran X) \rightleftarrows \coLiestar(\Ran X): \coChev.
\]

Unlike the case of $\Liestar$ and $\ComCoAlgstar$, for a co-Lie algebra $\mathfrak{g} \in \coLiestar(X)$,
\[
	\coChev(\mathfrak{g}) \in \ComAlgstar(\Ran X)
\]
doesn't necessarily live inside $\Factstar(X)$. However, we have the following
\begin{thm}[Theorem~\ref{thm:factorizability_coChev}] \label{thm:intro:factorizability_coChev}
	Restricted to the full subcategory $\coLiestar(X)^{\geq 1}$, where we are using the perverse $t$-structure on $X$, the functor $\coChev$ factors through $\Factstar$, i.e.
	\[
	\xymatrix{
		\coLiestar(X)^{\geq 1} \ar[dr]_{\coChev} \ar[rr]^{\coChev} && \ComAlgstar(\Ran X) \\
		& \Factstar(X) \ar@{^(->}[ur]
	}
	\]
\end{thm}

\subsubsection{Interaction between $\coChev$ and factorization homology} In~\cite{francis_chiral_2011}, it is proved that the functor of taking factorization homology
\[
	C^*_c: \Shv(\Ran X) \to \Vect
\]
commutes with $\Chev$. This is because $\Chev$ is computed as a colimit, and moreover, $C^*_c$ has the following two useful properties
\begin{enumerate}[\quad (i)]
	\item $C^*_c$ is symmetric monoidal with respect to the $\otimesstar$-monoidal structure on $\Shv(\Ran X)$ and the usual monoidal structure on $\Vect$, and
	
	\item $C^*_c$ is continuous.
\end{enumerate}

The functor $\coChev$, however, is constructed as a limit, so we need some extra conditions to make it behave nicely with $C^*_c$.

\begin{thm}[Theorem~\ref{thm:coChev_and_C^*_c(Ran)}] \label{thm:intro:coChev_and_C^*_c(Ran)}
Let $X$ be a proper scheme of pure dimension $d$ and $\mathfrak{g} \in \coLiestar(X)^{\geq d+1}$. Then we have a natural equivalence
\[
	C^*_c(\Ran X, \coChev \mathfrak{g}) \simeq \coChev(C^*_c(\Ran X, \mathfrak{g})).
\]
\end{thm}

\subsubsection{$\Chev$, $\coChev$ and Verdier duality} Unsurprisingly, the functors $\Chev$ and $\coChev$ mentioned above are linked via the Verdier duality functor on $\Ran X$.

\begin{thm}[Theorem~\ref{thm:Chev_coChev_and_D_Ran}] \label{thm:intro:Chev_coChev_and_D_Ran}
	Let $\mathfrak{g} \in \Liestar(X)^{\leq -1}$, where we are using the perverse $t$-structure on $X$. Then we have the following natural equivalence
	\[
		D_{\Ran X} \Chev \mathfrak{g} \simeq \coChev (D_X \mathfrak{g}).
	\]
\end{thm}

\begin{rmk}
The connectivity constraint $\Liestar(X)^{\leq -1}$ is, as we shall see, less strict than the connectivity constraint $\Liestar(X)^{\leq c_L}$ required by Theorem~\ref{thm:intro:Koszul_duality_connectivity_on_Ran}.
\end{rmk}

\begin{cor} \label{thm:intro:factorizability_D_Chev}
Let $\mathfrak{g} \in \Liestar(X)^{\leq c_L}$. Then
\[
	D_{\Ran X} \Chev \mathfrak{g} \simeq \coChev(D_X \mathfrak{g})
\]
is factorizable.
\end{cor}
\begin{proof}
This is a direct consequence of Theorem~\ref{thm:intro:factorizability_coChev} and Theorem~\ref{thm:intro:Chev_coChev_and_D_Ran}.
\end{proof}

\subsection{Relation to the Atiyah-Bott formula}
Our results could be used to simplify the second step of the proof of the Atiyah-Boot formula in two places, which we will sketch in~\S\ref{subsubsec:intro:factorizability_D_Chev} and~\S\ref{subsubsec:intro:Verdier_dual_linear_dual} below. A more detailed exposition will be given in~\S\ref{sec:application_Atiyah-Bott}.

\subsubsection{Factorizability of $D_{\Ran X} \Chev \mathfrak{a}$} \label{subsubsec:intro:factorizability_D_Chev} The initial observation is that the sheaf $\matheur{A}$ mentioned above lies in the essential image of $\Chev$, i.e.
\[
	\matheur{A} \simeq \Chev(\mathfrak{a}), \qquad\text{for some } \mathfrak{a} \in \Liestar(X)^{\leq c_L}.
\]
This is a direct result of Theorem~\ref{thm:intro:Koszul_duality_connectivity_on_Ran} and the fact that $\matheur{A}$ satisfies this connectivity constraint on the $\ComCoAlgstar$ side.

As mentioned above, we have a pairing
\[
	\matheur{A} \boxtimes \matheur{B} \to \delta_! \omega_{\Ran X},
\]
which induces a map
\[
	\matheur{B} \to D_{\Ran X} \matheur{A},
\]
compatible with the commutative algebra structures on both sides. Thus, we get a map
\[
	\matheur{B} \to D_{\Ran X}\Chev(\mathfrak{a}) \simeq \coChev(D_X \mathfrak{a}),
\]
which we want to be an equivalence. By construction, the LHS is factorizable. Corollary~\ref{thm:intro:factorizability_D_Chev} can be used to show that the RHS is also factorizable. Thus it suffices to show that they are isomorphic over $X$, which is now a local problem, and the same proof as in~\cite{gaitsgory_atiyah-bott_2015} applies.

\subsubsection{Verdier duality vs. linear dual.} \label{subsubsec:intro:Verdier_dual_linear_dual}
The results proved in this paper could also be used to give an alternative proof of the equivalence
\[
	C^*_c(\Ran X, D_{\Ran X} \matheur{A}) \simeq C^*_c(\Ran X, \matheur{A})^\vee.
\]
at~\eqref{eq:intro_verdier_homology}. Indeed, we have
\begin{align*}
	C^*_c(\Ran X, D_{\Ran X} \matheur{A}) 
	&\simeq C^*_c(\Ran X, D_{\Ran X} \Chev \mathfrak{a}) \\
	&\simeq C^*_c(\Ran X, \coChev D_X \mathfrak{a}) \tag{Theorem~\ref{thm:intro:Chev_coChev_and_D_Ran}} \\
	&\simeq \coChev(C^*_c(X, D_X \mathfrak{a})) \tag{Theorem~\ref{thm:intro:coChev_and_C^*_c(Ran)}} \\
	&\simeq \coChev(C^*_c(X, \mathfrak{a})^\vee) \\
	&\simeq \Chev(C^*_c(X, \mathfrak{a}))^\vee \tag{Theorem~\ref{thm:intro:Chev_coChev_and_D_Ran} for $X = \pt$} \\
	&\simeq C^*_c(\Ran X, \Chev \mathfrak{a})^\vee \tag{\cite[Proposition 6.3.6]{francis_chiral_2011}} \\
	&\simeq C^*_c(\Ran X, \matheur{A})^\vee
\end{align*}

\section{Preliminaries} \label{sec:prelims}
In this section, we will set up the language and conventions used throughout the paper. Since the material covered here are used in various places, the readers should feel free to skip it and backtrack when necessary.

The mathematical content in this section has already been treated elsewhere. Hence, results are stated without any proof, and we will do our best to provide the necessary references. It is important to note that it is not our aim to be exhaustive. Rather, we try to familiarize the readers with the various concepts and results used in the text, as well as to give pointers to the necessary references for the background materials.

\subsection{Notation and conventions}
\subsubsection{Category theory} We will use $\DGCat$ to denote the $(\infty, 1)$-category of stable infinity categories, $\DGCatpres$ to denote the full subcategory of $\DGCat$ consisting of presentable categories, and $\DGCatprescont$ the (non-full) subcategory of $\DGCatpres$ where we restrict to continuous functors, i.e. those commuting with colimits. $\Spc$ will be used to denote the category of spaces, or more precisely, $\infty$-groupoids.

The main references for this subject are~\cite{lurie_higher_2017} and~\cite{lurie_higher_2017-1}. For a slightly different point of view, see also~\cite{gaitsgory_study_2017}.

\subsubsection{Algebraic geometry} \label{subsubsec:ag_convention} Throughout this paper, $k$ will be an algebraically closed ground field. We will denote by $\Sch$ the $\infty$-category obtained from the ordinary category of separated schemes of finite type over $k$. All our schemes will be objects of $\Sch$. A scheme $X\in \Sch$ is said to be smooth if it is smooth over $k$.

In most cases, we will use the calligraphic font to denote prestacks, for eg. $\matheur{X}, \matheur{Y}$ etc., and the usual font to denote schemes, for eg. $X, Y$ etc.

\subsubsection{$t$-structures} \label{subsubsec:t-structure_convention} Let $\matheur{C}$ be a stable infinity category, equipped with a $t$-structure. Then we have the following diagram of adjoint functors
\[
\xymatrix{
	\matheur{C}^{\leq 0} \ardis[r]^>>>>>{i_{\leq 0}} & \ardis[l]^>>>>>{\tr_{\leq 0}} \matheur{C} \ardis[r]^<<<<<{\tr_{\geq 1}} & \ardis[l]^>>>>>{i_{\geq 1}} \matheur{C}^{\geq 1}
}
\]

We use $\tau_{\leq 0}$ and $\tau_{\geq 1}$ to denote
\[
\tau_{\leq 0} = i_{\leq 0} \circ \tr_{\leq 0}: \matheur{C} \to \matheur{C}
\]
and
\[
\tau_{\geq 1} = i_{\geq 1} \circ \tr_{\geq 1}: \matheur{C} \to \matheur{C}
\]
respectively.

Shifts of these functors, for e.g. $\tau_{\geq n}$ and $\tau_{\leq n}$, are defined in the obvious ways.

\subsection{Prestacks} 
The theory of sheaves on prestacks has been developed in~\cite{gaitsgory_weils_2014} and~\cite{gaitsgory_atiyah-bott_2015}. In this subsection and the next, we will give a brief review of this theory, including  the definition of the category of sheaves as well as various pull and push functors. We will state them as facts, without any proof, which (unless otherwise specified), could all be found in~\cite{gaitsgory_atiyah-bott_2015}.

\subsubsection{} A prestack is a contravariant functor from $\Sch$ to $\Spc$, i.e. a prestack $\matheur{Y}$ is a functor
\[
	\matheur{Y}: \Sch^{\op} \to \Spc.
\]
Let $\PreStk$ be the $\infty$-category of prestacks. Then by Yoneda's lemma, we have a fully-faithful embedding
\[
	\Sch \hookrightarrow \PreStk.
\]

\subsubsection{Properties of prestacks} Due to categorical reasons, any prestack $\matheur{Y}$ can be written as a colimit of schemes
\[
	\matheur{Y} \simeq \colim_{i\in I} Y_i.
\]

\subsubsection{} A prestack is said to be is a pseudo-scheme if it could be written as a colimit of schemes, where all morphisms are proper. 

\subsubsection{} A prestack is pseudo-proper if it could be written as a colimit of proper schemes. It is straightforward to see that pseudo-proper prestacks are pseudo-schemes.

\subsubsection{} A prestack is said to be finitary if it could be expressed as a finite colimit of schemes. 

\subsubsection{} We also have relative versions of the definitions above in an obvious manner. Namely, we can speak of a morphism $f: \matheur{Y} \to S$, where $\matheur{Y}$ is a prestack and $S$ is a scheme, being pseudo-schematic (resp. pseudo-proper, finitary).

\subsubsection{} More generally, a morphism
\[
	f: \matheur{Y}_1 \to \matheur{Y}_2
\]
is said to be pseudo-schematic (resp. pseudo-proper, finitary) if for any scheme $S$, equipped with a morphism $S \to \matheur{Y}_2$, the morphism $f_S$ in the following pull-back diagram
\[
\xymatrix{
	S \times_{\matheur{Y}_2} \matheur{Y}_1 \ar[d]_{f_S} \ar[r] & \matheur{Y}_1 \ar[d] \\
	S \ar[r] & \matheur{Y}_2
}
\]
is pseudo-schematic (resp. pseudo-proper, finitary).

\subsection{Sheaves on prestacks}
As we mentioned above, proofs of all the results in mentioned in this section, unless otherwise specified, could be found in~\cite{gaitsgory_atiyah-bott_2015}.

\subsubsection{Sheaves on schemes} We will adopt the same conventions as in~\cite{gaitsgory_atiyah-bott_2015}, except that for simplicity, we will restrict ourselves to the ``constructible setting.'' Namely, for a scheme $S$,

\begin{enumerate}[(i)]
	\item when the ground field is $\mathbb{C}$, and $\Lambda$ is an arbitrary field of characteristic 0, we take $\Shv(S)$ to be the ind-completion of the category of constructible sheaves on $S$ with $\Lambda$-coefficients.
	
	\item for any ground field $k$ in general, and $\Lambda = \Ql, \Qlbar$ with $\ell \neq \car k$, we take $\Shv(S)$ to be the ind-completion of the category of constructible $\ell$-adic sheaves on $S$ with $\Lambda$-coefficients. See also~\cite[\S4]{gaitsgory_weils_2014}, \cite{liu_enhanced_2012}, and~\cite{liu_enhanced_2014}.
\end{enumerate}

The theory of sheaves on schemes is equipped with the various pairs of adjoint functors
\[
	f_! \dashv f^! \qquad\text{and}\qquad f^* \dashv f_*
\]
for any morphism
\[
	f: S_1 \to S_2
\]
between schemes. Moreover, we have box-product $\boxtimes$ as well as $\otimes$ and $\otimesshriek$. 

\subsubsection{} Throughout the text, we will use the perverse $t$-structure on $\Shv(S)$, when $S$ is a scheme.

\subsubsection{} We will also use $\Vect$ to denote the category of sheaves on a point, i.e. $\Vect$ denotes the (infinity derived) category of chain complexes in vector spaces over $\Lambda$.

\subsubsection{Sheaves on prestacks} For a prestack $\matheur{Y}$, the category $\Shv(\matheur{Y})$ is defined by
\[
	\Shv(\matheur{Y}) = \lim _{S\in (\Sch_{/\matheur{Y}}^{\op})} \Shv(S),
\]
where the transition functor we use is the !-pullback.

Informally speaking, an object $\matheur{F} \in \Shv(\matheur{Y})$ is the same as the following data
\begin{enumerate}[(i)]
	\item a sheaf $\matheur{F}_{S, y} \in \Shv(S)$ for each $S\in \Sch$ and $y: S \to \matheur{Y}$ (i.e. $y\in \matheur{Y}(S)$), and
	\item an equivalence of sheaves $\matheur{F}_{S', f(y)} \to f^! \matheur{F}_{S, y}$ for each morphism of schemes $f: S' \to S$.
\end{enumerate}
Moreover, we require that this assignment satisfies a homotopy-coherent system of compatibilities.

\subsubsection{} More formally, one can define $\Shv(\matheur{Y})$ as the right Kan extension of
\[
	\Shv: \Sch^{\op} \to \DGCatprescont
\]
along the Yoneda embedding
\[
	\Sch^{\op} \hookrightarrow \PreStk^{\op}.
\]

Thus, by formal reasons, the functor
\[
	\Shv: \PreStk^{\op} \to \DGCatprescont
\]
preserves limits. In other words, we have
\[
	\Shv(\colim_i \matheur{Y}_i) \simeq \lim_i \Shv(\matheur{Y}_i).
\]

In particular, if a prestack
\[
	\matheur{Y} \simeq \colim_{i\in I} Y_i
\]
is a colimit of schemes, then
\[
	\Shv(\matheur{Y}) \simeq \lim_{i\in I} \Shv(Y_i).
\]

\subsubsection{} Now, if we replace all the transition functors by their left adjoints, namely the $!$-pushforward, then we have a diagram
\[
	I^\op \to \DGCatprescont,
\]
and we have a natural equivalence
\[
	\Shv(\matheur{Y}) \simeq \colim_{i\in I^\op} \Shv(Y_i)
\]
where the colimit is taken inside $\DGCatprescont$.

\subsubsection{} Let
\[
	\matheur{Y} = \colim_i Y_i
\]
be a prestack, and denote
\[
	\ins_i: Y_i \to \matheur{Y}
\]
the canonical map. Then, for any sheaf $\matheur{F} \in \Shv(\matheur{Y})$, we have the following natural equivalence
\[
	\matheur{F} \simeq \colim_i \ins_{i!} \ins_i^! \matheur{F} \teq\label{eq:sheaves_colimits_rep}
\]

\subsubsection{$f_! \dashv f^!$} 
Let
\[
	f: \matheur{Y}_1 \to \matheur{Y}_2
\]
be a morphism between prestacks. Then by restriction, we get a functor
\[
	f^!: \Shv(\matheur{Y}_2) \to \Shv(\matheur{Y}_1),
\]
which commutes with both limits and colimits. In particular, $f^!$ admits a left adjoint $f_!$.\footnote{It also admits a right adjoint. However, we do not make use of it in this paper.}

The functor $f_!$ is generally not computable. However, there are a couple of cases where it is.

\subsubsection{} The first instance is when the target of $f$ is a scheme
\[
	f: \matheur{Y} \to S,
\]
and suppose that
\[
	\matheur{Y} \simeq \colim_i Y_i.
\]
Then, by~\eqref{eq:sheaves_colimits_rep}, we have
\[
	f_! \matheur{F} \simeq \colim f_! \ins_{i!}\ins_i^! \matheur{F} \simeq \colim f_{i!} \ins_i^! \matheur{F}.
\]
where
\[
	f_i: Y_i \to \matheur{Y} \to S
\]
is just a morphism between schemes.

\subsubsection{} The second case is where $f$ is pseudo-proper, then $f_!$ satisfies the base change theorem with respect to the $(-)^!$-pullback. Namely, for any pull-back diagram of prestacks
\[
\xymatrix{
	\matheur{Y}'_1 \ar[d]_{f} \ar[r]^g & \matheur{Y}_1 \ar[d]_{f} \\
	\matheur{Y}'_2 \ar[r]^g & \matheur{Y}_2
}
\]
and any sheaf $\matheur{F} \in \Shv(\matheur{Y})$, we have a natural equivalence
\[
	g^! f_! \matheur{F} \simeq f_! g^! \matheur{F}.
\]

Thus, in particular, if we have a pull-back diagram
\[
\xymatrix{
	S \times_{\matheur{Y}_2} \matheur{Y}_1 \ar[d]_{f_S} \ar[r]^{i_S} & \matheur{Y}_1 \ar[d]_f \\
	S \ar[r]^{i_S} & \matheur{Y}_2
}
\]
where $S$ is a scheme, then
\[
	i_S^! f_! \matheur{F} \simeq f_{S!} i_S^!\matheur{F}
\]
and as discussed above, $f_{S!}$ could be computed as an explicit colimit.

\subsubsection{} Let $\matheur{F} \in \Shv(\matheur{Y})$. Then we denote by
\[
	C^*_c(\matheur{Y}, \matheur{F}) = s_! \matheur{F},
\]
where
\[
	s: \matheur{Y} \to \Spec k
\]
is the structural map of $\matheur{Y}$ to a point.

\subsubsection{} In case where $\matheur{F} \simeq \omega_{\matheur{Y}}$ is the dualizing sheaf on $\matheur{Y}$ (characterized by the fact that its $(-)^!$-pullback to any scheme is the dualizing sheaf on that scheme), then we write
\[
	C_*(\matheur{Y}) = C^*_c(\matheur{Y}, \omega_{\matheur{Y}}),
\]
and
\[
	C_*^{\red}(\matheur{Y}) = \Fib(C_*(\matheur{Y}) \to \Lambda).
\]

\subsubsection{$f^* \dashv f_*$} When
\[
	f: \matheur{Y}_1 \to \matheur{Y}_2
\]
is a schematic morphism between prestacks, one can also define a pair of adjoint functors (see~\cite{gaitsgory_atiyah-bott_2015} where the functor $f_*$ is defined, and~\cite{ho_free_2017} where the adjunction is constructed)
\[
	f^*: \Shv(\matheur{Y}_2) \rightleftarrows \Shv(\matheur{Y}_1): f_*.
\]

\subsubsection{} The behavior of $f_*$ is easy to describe, due to the fact that $f_*$ satisfies the base change theorem with respect to the $(-)^!$-pullback functor. Namely, suppose $\matheur{F} \in \Shv(\matheur{Y}_1)$ and we have a pullback square where $S_2$ (and hence, $S_1$) is a scheme
\[
\xymatrix{
	S_1 \ar[r]^g \ar[d]_{f_S} & \matheur{Y}_1 \ar[d]_{f} \\
	S_2 \ar[r]^g & \matheur{Y}_2
}
\]
Then, the pullback could be described in classical terms, since
\[
	g^! f_* \matheur{F} \simeq f_{S*} g^! \matheur{F}, \teq\label{eq:base_change_lower*_upper!}
\]
where $f_S$ is just a morphism between schemes.

\subsubsection{} The functor $f^*$ is slightly more complicated to describe. However, when
\[
	f: \matheur{Y}_1 \to \matheur{Y}_2
\]
is \etale{}, which is the case where we need, we have a natural equivalence (see~\cite[Prop. 2.7.3]{ho_free_2017})
\[
	f^! \simeq f^*. \teq\label{eq:etale_!_*_pullback_equiv}
\]

\subsubsection{} We will also need the following fact in the definition of commutative factorizable co-algebras: let
\[
\xymatrix{
	\matheur{U} \ar[r]^f & \matheur{Z} \ar[r]^g & \matheur{X}
}
\]
be morphisms between prestacks, where $g$ is finitary pseudo-proper, $f$ and $h = g\circ f$ are schematic. Then we have a natural equivalence (see~\cite[Prop. 2.10.4]{ho_free_2017})
\[
	g_! \circ f_* \simeq (g\circ f)_* \simeq h_*. \teq\label{eq:pushforward_!circ*_equiv_*}
\]

\subsubsection{Monoidal structure} The theory of sheaves on prestacks discussed so far naturally inherits the box-tensor structure from the theory of sheaves on schemes. Namely, let $\matheur{F}_i\in \Shv(\matheur{Y}_i)$ where $\matheur{Y}_i$'s are prestacks, for $i=1, 2$. Then, for any pair of schemes $S_1, S_2$ equipped with maps
\[
	f_i: S_i\to \matheur{Y}_i,
\]
we have 
\[
	(f_1 \times f_2)^! (\matheur{F}_1 \boxtimes \matheur{F}_2) \simeq f_1^! \matheur{F}_1 \boxtimes f_2^! \matheur{F}_2.
\]

Pulling back along the diagonal
\[
	\delta: \matheur{Y} \to \matheur{Y} \times \matheur{Y}
\]
for any prestack $\matheur{Y}$, we get the $\otimesshriek$-symmetric monoidal structure on $\matheur{Y}$ in the usual way. More explicitly, for $\matheur{F}_1, \matheur{F}_2 \in \Shv(\matheur{Y})$, we define
\[
	\matheur{F}_1 \otimesshriek\matheur{F}_2 = \delta^! (\matheur{F}_1 \boxtimes \matheur{F}_2).
\]

\subsection{The $\Ran$ space/prestack}
The $\Ran$ space (or more precisely, prestack) of a scheme plays a central role in this paper. The $\Ran$ space, along with various objects on it, was first studied in the seminal book~\cite{beilinson_chiral_2004} in the case of curves, and was generalized to higher dimensions in~\cite{francis_chiral_2011}. In what follows, we will quickly review the main definitions and results. For proofs, unless otherwise specified, we refer the reader to~\cite{gaitsgory_atiyah-bott_2015} and~\cite{francis_chiral_2011}. The topologically inclined reader could also find an intuitive introduction in~\cite[\S1]{ho_free_2017}.

\subsubsection{} For a scheme $X \in \Sch$, we will use $\Ran X$ to denote the following prestack: for each scheme $S \in \Sch$,
\[
	(\Ran X)(S) = \{\text{non-empty finite subsets of } X(S)\}
\]
Alternatively, one has
\[
	\Ran X \simeq \colim_{I\in \fSet^{\surj, \op}} X^I
\]
where $\fSet^{\surj}$ denotes the category of non-empty finite sets, where morphisms are surjections.

Using the fact that $X$ is separated, one sees easily that $\Ran X$ is a pseudo-scheme. Moreover, when $X$ is proper, $\Ran X$ is pseudo-proper.

\subsubsection{The $\otimesstar$ monoidal structure}
There is a special monoidal structure on $\Ran X$ which we will use throughout the text: the $\otimesstar$-monoidal structure.

Consider the following map
\[
\xymatrix{
	\union: \Ran X \times \Ran X \to \Ran X
}
\]
given by the union of non-empty finite subsets of $X$. One can check that $\union$ is finitary pseudo-proper. Given two sheaves $\matheur{F}, \matheur{G} \in \Shv(\Ran X)$, we define
\[
	\matheur{F} \otimesstar\matheur{G} = \union_!(\matheur{F} \boxtimes \matheur{G}).
\]

This defines the $\otimesstar$-monoidal structure on $\Shv(\Ran X)$.

\subsubsection{} Since $\union$ is pseudo-proper, $\otimesstar$ has an easy presentation. Namely, for
\[
	\matheur{F}_1, \matheur{F}_2, \dots, \matheur{F}_k \in \Shv(\Ran X),
\]
and any non-empty finite set $I$, we have the following
\[
	(\matheur{F}_1 \otimesstar \matheur{F}_2 \otimesstar\cdots \otimesstar \matheur{F}_k)|_{\oversetsupscript{X}{\circ}{I}} \simeq \bigoplus_{I = \bigcup_{i=1}^k I_i} \Delta^!_{\sqcup_{i=1}^k I_i \surjects \cup_{i=1}^k I_i}(\matheur{F}_1 \boxtimes \cdots \boxtimes\matheur{F}_k)|_{\oversetsupscript{X}{\circ}{I}} \teq \label{eq:otimesstar_explicit_formula}
\]
where $\oversetsupscript{X}{\circ}{I}$ denotes the open subscheme of $X^I$ given by the condition that no two ``coordinates'' are equal, and where
\[
	\Delta_{\sqcup_{i=1}^k I_i \surjects \cup_{i=1}^k I_i}: X^I \hookrightarrow \prod_i X^{I_i}
\]
is the map induced by the surjection
\[
	\bigsqcup_{i=1}^k I_i \surjects \bigcup_{i=1}^k I_i \simeq I.
\]

\subsubsection{Factorizable sheaves}
Using the $\otimesstar$-monoidal structure on $\Shv(\Ran X)$, one can talk about various types of algebras/coalgebras in $\Shv(\Ran X)$. The ones that are of importance to us in this papers are
\[
	\ComAlgstar(\Ran X),\quad \Liestar(\Ran X),\quad \ComCoAlgstar(\Ran X),\quad  \coLiestar(\Ran X).
\]
As the name suggests, these are used, respectively, to denote the categories of commutative algebras, Lie algebras, co-commutative co-algebras and co-Lie co-algebras in $\Shv(\Ran X)$ with respect to the $\otimesstar$-monoidal structure defined above.

\subsubsection{} We use $\Liestar(X)$ and $\coLiestar(X)$ to denote the full subcategories of $\Liestar(\Ran X)$ and $\coLiestar(\Ran X)$ respectively, consisting of objects whose supports are inside the diagonal
\[
	\ins_X: X \hookrightarrow \Ran X
\]
of $\Ran X$.

\subsubsection{} Let
\[
	j: (\Ran X)^n_\disj \to (\Ran X)^n
\]
where $(\Ran X)^n_{\disj}$ is the open sub-prestack of $(\Ran X)^n$ defined by the following condition: for each scheme $S$, $(\Ran X)^n(S)$ consists of $n$ non-empty subsets of $X(S)$, whose graphs are pair-wise disjoint.

\subsubsection{} Let
\[
	\matheur{A} \in \ComCoAlgstar(\Ran X).
\]
Then, by definition, we have the following map (which is the co-multiplication of the commutative co-algebra structure)
\[
	\matheur{A} \to \matheur{A} \otimesstar \matheur{A} \otimesstar \cdots \otimesstar \matheur{A} \simeq \union_!(\matheur{A} \boxtimes \cdots \boxtimes \matheur{A}).
\]
Using the the unit map of the adjunction $j^* \dashv j_*$, we get the following map
\[ 
	\union_!(\matheur{A}\boxtimes \cdots \boxtimes \matheur{A}) \to \union_! j_* j^* (\matheur{A} \boxtimes \cdots \boxtimes \matheur{A}) \simeq (\union \circ j)_* j^!(\matheur{A} \boxtimes \cdots \boxtimes \matheur{A}),
\]
where for the equivalence, we made use of~\eqref{eq:etale_!_*_pullback_equiv} and~\eqref{eq:pushforward_!circ*_equiv_*}.

Altogether, we get a map
\[
	\matheur{A} \to (\union\circ j)_* j^! (\matheur{A} \boxtimes \cdots \boxtimes \matheur{A})
\]
and hence, by adjunction and~\eqref{eq:etale_!_*_pullback_equiv}, we get a map

\[
	j^! \union^! \matheur{A} \to j^! (\matheur{A}\boxtimes \cdots \boxtimes \matheur{A}). \teq\label{eq:commutative_cofac_map}
\]

\begin{defn}
	$\matheur{A}$ is a commutative factorization algebra if the map~\eqref{eq:commutative_cofac_map} is an equivalence for all $n$'s. We use $\coFactstar(X)$ to denote the full subcategory of $\ComAlgstar(\Ran X)$ consisting of co-commutative factorization co-algebras.
\end{defn}

\subsubsection{} Let
\[
	\matheur{B} \in \ComAlgstar(\Ran X).
\]
Then, by definition, we have the following map (which is the multiplication of the commutative algebra structure)
\[
	\union_!(\matheur{B} \boxtimes \matheur{B} \boxtimes \cdots \boxtimes \matheur{B}) \simeq \matheur{B} \otimesstar \matheur{B} \otimesstar \cdots \otimesstar \matheur{B} \to \matheur{B}.
\]
This induces the following map of sheaves
\[
	\matheur{B} \boxtimes \cdots \boxtimes \matheur{B} \to \union^! \matheur{B}
\]
on $(\Ran X)^n$, and hence, a map of sheaves
\[
	j^! (\matheur{B} \boxtimes \cdots \boxtimes \matheur{B}) \to j^!\union^! \matheur{B}. \teq\label{eq:commutative_fac_map}
\]
on $(\Ran X)^n_\disj$.

\begin{defn}
	$\matheur{B}$ is a commutative factorization algebra if the map~\eqref{eq:commutative_fac_map} is an equivalence for all $n$'s. We use $\Factstar(X)$ to denote the full subcategory of $\ComAlgstar(\Ran X)$ consisting of commutative factorization algebras.
\end{defn}

\subsection{Koszul duality}
In this subsection, we will quickly review various concepts and results in the theory of Koszul duality that are relevant to us. This theory, initially developed in~\cite{quillen_rational_1969}, illuminates the duality between co-commutative co-algebras and Lie algebras. It was further developed and generalized in the operadic setting in~\cite{ginzburg_koszul_1994}. In the chiral/factorizable setting, the paper~\cite{francis_chiral_2011} provides us with necessary technical tools and language to carry out many topological arguments in the context of algebraic geometry. The results and definitions we review below could be found in~\cite{francis_chiral_2011} and~\cite{gaitsgory_study_2017}. 

\subsubsection{Symmetric sequences} Let $\Vect^{\Sigma}$ denote the category of symmetric sequences. Namely, its objects are collections
\[
	\matheur{O} = \{\matheur{O}(n), n\geq 1\},
\]
where each $\matheur{O}(n)$ is an object of $\Vect$, acted on by the symmetric group $\Sigma_n$.

The infinity category $\Vect^{\Sigma}$ is equipped with a natural monoidal structure, which we denote by $\star$, and which makes the functor
\[
	\Vect^{\Sigma} \to \Fun(\Vect, \Vect)
\]
given by the following formula
\[
	\matheur{O} \star V = \bigoplus_n (\matheur{O}(n)\otimes V^{\otimes n})_{\Sigma_n}
\]
symmetric monoidal.

\subsubsection{Operads and co-operads} By an operad (resp. co-operad), we will mean an augmented associative algebra (resp. co-algebra) object in $\Vect^{\Sigma}$, with respect to the monoidal structure described above.  We use $\Op$ (resp. $\coOp$) to denote the categories of operads (resp. co-operads).

In general, the Bar and coBar construction gives us the following pair of adjoint functors
\[
	\BarO: \Op \rightleftarrows \coOp: \coBarP.
\]
For an operad $\matheur{O}$ (resp. co-operad $\matheur{P}$), we also use $\matheur{O}^\vee$ (resp. $\matheur{P}^\vee$) to denote $\BarO(\matheur{O})$ (resp. $\coBarP(\matheur{P})$).

\begin{rmk}
In what follows, we will adopt the following convention: all our operads/co-operads will have the property that the augmentation map is an equivalence when restricted to $\matheur{O}(1)$ (resp. $\matheur{P}(1)$). And under this restriction, one can show that the following unit map
\[
	\matheur{O} \to \coBarP\circ \BarO(\matheur{O}),
\]
or in a slightly different notation
\[
	\matheur{O} \to (\matheur{O}^\vee)^\vee,
\]
is an equivalence.
	
\end{rmk}

\subsubsection{Algebras and co-algebras} Let $\matheur{C}$ be a stable presentable symmetric monoidal $\infty$-category compatibly tensored over $\Vect$. Then, an operad $\matheur{O}$ (resp. co-operad $\matheur{P}$) naturally defines a monad (resp. comonad) on $\matheur{C}$. 

Thus, for an operad $\matheur{O}$ (resp. co-operad $\matheur{P}$), one can talk about the category of algebras $\matheur{O}\Palg(\matheur{C})$ (resp. co-algebras $\matheur{P}\Pcoalg(\matheur{C})$) in $\matheur{C}$ with respect to the operad $\matheur{O}$ (resp. co-operad $\matheur{P}$).

As usual (as for any augmented monad), one has the following pairs of adjoint functors
\[
	\Free_{\matheur{O}}: \matheur{C} \rightleftarrows \matheur{O}\Palg(\matheur{C}):\oblv_{\matheur{O}} \qquad \text{and}\qquad \BarO_{\matheur{O}}: \matheur{O}\Palg(\matheur{C}) \rightleftarrows \matheur{C} :\triv_{\matheur{O}}
\]
for an operad $\matheur{O}$, and similarly, the following pairs of adjoint functors
\[
	\oblv_{\matheur{P}}: \matheur{P}\Pcoalg(\matheur{C}) \rightleftarrows \matheur{C} :\coFree_{\matheur{P}} \qquad\text{and}\qquad \cotriv_{\matheur{P}}: \matheur{C} \rightleftarrows \matheur{P}\Pcoalg(\matheur{C}) :\coBarP_{\matheur{P}}
\]
for a co-operad $\matheur{P}$.

\subsubsection{Koszul duality}
The functors mentioned above could be lifted to get a pair of adjoint functors
\[
	\BarO^\enh: \matheur{O}\Palg(\matheur{C}) \rightleftarrows \matheur{P}\Pcoalg(\matheur{C}): \coBarP^\enh \teq \label{eq:adjunction_Koszul_dual}
\]
where $\matheur{P} = \matheur{O}^\vee$ and
\[
	\oblv_{\matheur{P}}\circ \BarO_\matheur{O}^\enh \simeq \BarO_\matheur{O} \qquad\text{and}\qquad \oblv_{\matheur{O}}\circ \coBarP_\matheur{P}^\enh \simeq \coBarP_\matheur{P}.
\]

\subsubsection{Turning Koszul duality into an equivalence} In general, the pair of adjoint functors at~\eqref{eq:adjunction_Koszul_dual} is not an equivalence. One of the main achievements of~\cite{francis_chiral_2011} is to formulate a precise sufficient condition on the base category $\matheur{C}$, namely the pro-nilpotent condition,\footnote{The interested reader could read more about this in~\cite{francis_chiral_2011}, since we do not need this fact in the current work.} which turns~\eqref{eq:adjunction_Koszul_dual} into an equivalence.

One of the main technical points of our paper is to prove another case where Koszul duality is still an equivalence, even when the categories involved are not pro-nilpotent.

The two main instances of Koszul duality that are important in this paper are the duality between $\Lie$-algebras and $\ComCoAlg$-algebras, and $\coLie$-algebras and $\ComAlg$-algebras.

\subsubsection{The case of $\Lie$ and $\ComCoAlg$}
We have the following equivalence of co-operads (see~\cite{francis_chiral_2011}):
\[
\Lie^\vee \simeq \ComCoAlg[1],
\]
where
\[
\ComCoAlg[1](n) \simeq k[n-1]
\]
is equipped with the sign action of the symmetric group $\Sigma_n$. Here, $[n]$ denotes cohomological shift to the left by $n$.

Equivalently, the functor
\[
[1]: \matheur{C} \to \matheur{C}
\]
gives rise to an equivalence of categories
\[
[1]: \ComCoAlg[1](\matheur{C}) \simeq \ComCoAlg(\matheur{C}).
\]

This gives us the following diagram
\[
\xymatrix{
	\Lie(\matheur{C}) \ardissp[ddrr]^{\Chev} \ardis[dd]^{[1]} \ardis[rr]^<<<<<<<<<<<<{\BarO_{\Lie}} && \ardis[ll]^<<<<<<<<<<<{\coBarP_{\ComCoAlg[1]}} \ComCoAlg[1](\matheur{C}) \ardis[dd]^{[1]} \\ \\
	\Lie[-1](\matheur{C}) \ardis[uu]^{[-1]} \ardis[rr]^{\BarO_{\Lie[-1]}} && \ardis[ll]^{\coBarP_{\ComCoAlg}} \ComCoAlg(\matheur{C}) \ardis[uu]^{[-1]} \ar[uull]^{\Prim[-1]}
}
\]

We usually use $\Chev$ to denote 
\[
	\Chev \simeq [1]\circ \BarO_{\Lie} \simeq \BarO_{\Lie[-1]} \circ [1] \teq\label{eq:Chev_and_Bar}
\]
and $\Prim[-1]$ to denote
\[
	\Prim[-1] \simeq \coBarP_{\ComCoAlg[1]}\circ [-1] \simeq [-1]\circ \coBarP_{\ComCoAlg}. \teq\label{eq:Prim[-1]_and_coBar}
\]

\subsubsection{The case of $\coLie$ and $\ComAlg$} \label{subsubsec:prelim_Koszul_coLie_ComAlg} Dually, we have the following equivalence of co-operads
\[
\ComAlg^\vee \simeq \coLie[1],
\]
and similar to the above, the functor
\[
[1]: \matheur{C} \to \matheur{C}
\]
gives rise to an equivalence of categories
\[
[1]: \coLie[1](\matheur{C}) \simeq \coLie(\matheur{C}).
\]

\subsubsection{} This gives us the following diagram
\[
\xymatrix{
	\ComAlg(\matheur{C}) \ardissp[ddrr]^{\coPrim[1]} \ardis[dd]^{[1]} \ardis[rr]^<<<<<<<<<<<<<{\BarO_{\ComAlg}} && \ardis[ll]^>>>>>>>>>>>>>{\coBarP_{\coLie[1]}} \coLie[1](\matheur{C}) \ardis[dd]^{[1]} \\ \\
	\ComAlg[-1](\matheur{C}) \ardis[uu]^{[-1]} \ardis[rr]^>>>>>>>>>>>{\BarO_{\ComAlg[-1]}} && \ardis[ll]^<<<<<<<<<<<{\coBarP_{\coLie}} \coLie(\matheur{C}) \ardis[uu]^{[-1]} \ar[uull]^{\coChev}
}
\]

As above, we usually use $\coChev$ to denote
\[
	\coChev = [-1] \circ \coBarP_{\coLie} \simeq \coBarP_{\coLie[1]} \circ [-1]
\]
and $\coPrim[1]$ to denote
\[
	\coPrim[1] = [1] \circ \BarO_{\ComAlg} \simeq \BarO_{\ComAlg[-1]} \circ [1].
\]

\section{Turning Koszul duality into an equivalence}
The goal of this section is to prove Theorem~\ref{thm:intro:Koszul_duality_connectivity_on_Ran}. We will start with Theorem~\ref{thm:equivalence_koszul_vec_case}, which examines the special case where $X$ is just a point, i.e. $\Shv(\Ran X) \simeq \Shv(X) \simeq \Vect$, and prove that Koszul duality induces a natural equivalence of categories
\[
	\Chev :\Lie(\Vect^{\leq -1}) \simeq \ComCoAlg(\Vect^{\leq -2}): \Prim[-1].
\]

Note that this is a classical result of Quillen~\cite{quillen_rational_1969}, and our proof could be viewed as a recast of his under the light of higher algebra. This point of view allows us to generalize the result to the more general case of interest. Note also that this case is not strictly needed in the proof of the general case. We do, however, recommend the reader to first read it before moving on to the proof of Theorem~\ref{thm:Koszul_duality_connectivity_on_Ran} since it contains all the essential points without the complicated notation employed in the general case to deal with the combinatorics of the $\Ran$ space.

\subsection{The case of $\Lie$- and $\ComCoAlg$-algebras inside $\Vect$} \label{subsec:Koszul_equivalence_Lie_ComCoAlg_in_Vect}
We will now prove the following
\begin{thm} \label{thm:equivalence_koszul_vec_case}
	$\Chev$ and $\Prim[-1]$ give rise to a pair of mutually inverse functors
	\[
		\Chev: \Lie(\Vect^{\leq -1}) \rightleftarrows \ComCoAlg(\Vect^{\leq -2}): \Prim[-1]
	\]
\end{thm}

\begin{rmk} \label{rmk:vec_case_land_in_correct_category}
Since $\Chev$ is defined as a colimit, it is easy to see that $\Chev|_{\Lie(\Vect^{\leq -1})}$ lands in the correct subcategory cut out by the connectivity assumption $\Vect^{\leq -2}$ (the extra shift to the left is due to~\eqref{eq:Chev_and_Bar}). It is, however, not a priori obvious for $\Prim[-1]$, being defined as a limit. Nonetheless, this fact is a direct consequence of Lemma~\ref{lem:vec_case_stability} and Corollary~\ref{cor:vec_case_from_coBar_to_coBarm}.
\end{rmk}

\begin{rmk}
Unless otherwise specified, when it makes sense our functors will be automatically restricted to the subcategories with the appropriate connectivity conditions. For example, we will write $\Chev$ instead of $\Chev|_{\Lie(\Vect^{\leq -1})}$ in most cases.
\end{rmk}

\begin{rmk}
Note that Theorem~\ref{thm:equivalence_koszul_vec_case} can be proved more generally for a presentable symmetric monoidal stable infinity category with a $t$-structure satisfying some mild properties. The pair of operad and co-operad $\Lie$ and $\ComCoAlg$ could also be made more general. See Remarks~\ref{rmk:general_category} and~\ref{rmk:general_operads}.
\end{rmk}

\subsubsection{}
To prove that $\Chev$ and $\Prim[-1]$ are mutually inverse functors, it suffices to show that the left adjoint functor, $\Chev$, is fully-faithful, and the right adjoint functor, $\Prim[-1]$ is conservative. We start with the following result, whose proof is carried out in \S\ref{subsubsec:proof_lem_Prim[-1]_is_good_Vec_case} after some preparation.

\begin{lem} \label{lem:Prim[-1]_is_good_Vec_case} The functor $\Prim[-1]|_{\ComCoAlg(\Vect^{\leq -2})}$ satisfies the following conditions
	\begin{enumerate}[(i)]
		\item \label{lem:Prim[-1]_is_good_Vec_case:sifted_colims} $\Prim[-1]$ commutes with sifted colimits.
		\item \label{lem:Prim[-1]_is_good_Vec_case:monad_isom} The natural map
		\[
			\Free_{\Lie} \to \Prim[-1]\circ \triv_{\ComCoAlg}
		\]
		is an equivalence.
	\end{enumerate}
\end{lem}

As in~\cite[\S 4.1.8]{francis_chiral_2011}, this immediately implies the following corollary. For the sake of completeness, we include the proof here.
\begin{cor} \label{cor:vect_case_Chev_fully_faithful}
	$\Chev|_{\Lie(\Vect^{\leq -1})}$ is fully faithful.
\end{cor}
\begin{proof}
It suffices to show that the unit map
\[
	\id \to \Prim[-1]\circ \Chev
\]
is an equivalence. Since $\Prim[-1]$ commutes with sifted colimits by part~\eqref{lem:Prim[-1]_is_good_Vec_case:sifted_colims} of Lemma~\ref{lem:Prim[-1]_is_good_Vec_case}, it suffices to show that the following is an equivalence
\[
	\Free_{\Lie} \to \Prim[-1]\circ \Chev \circ \Free_{\Lie},
\]
since any $\Lie$-algebra could be written as a sifted colimit of the free ones.\footnote{This fact applies to the category of algebras over any operad in general.}
However, we know that (even without the connectivity condition)
\[
	\Chev\circ \Free_{\Lie} \simeq \triv_{\ComCoAlg}
\]
and hence, it suffices to show that
\[
	\Free_{\Lie} \to \Prim[-1]\circ \triv_{\ComCoAlg}.
\]
But now, we are done due to part~\eqref{lem:Prim[-1]_is_good_Vec_case:monad_isom} of Lemma~\ref{lem:Prim[-1]_is_good_Vec_case}.
\end{proof}

\subsubsection{} Before proving Lemma~\ref{lem:Prim[-1]_is_good_Vec_case}, we start with a couple of preliminary observations. In essence, the lemma is a statement about commuting limits and colimits. In a stable infinity category, if, for instance, the limit is a finite one, then one can always do that. In our situation, $\coBarP$ causes troubles because it is defined as an infinite limit. 

The main idea of the proof is that when
\[
	c\in \ComCoAlg(\Vect^{\leq -2}),
\]
then even though
\[
	\coBarP_\matheur{\ComCoAlg}(c)
\]
is computed as an infinite limit, each of its cohomological degrees will be controlled by only finitely many of terms in the limit.

\subsubsection{} For brevity's sake, we will use $\matheur{P}$ to denote the co-operad $\ComCoAlg$. Recall that in general, for any
\[
	c\in \ComCoAlg(\Vect),
\]
we have
\[
	\coBarP_{\matheur{P}}(c) = \Tot(\coBarP^\bullet_{\matheur{P}}(c))
\]
where $\coBarP^\bullet_\matheur{P}(c)$ is a co-simplicial object.

Let
\[
	\coBarP_{\matheur{P}}^n(c) = \Tot(\coBarP_\matheur{P}^\bullet(c)|_{\Delta^{\leq n}})
\]
be the limit over the restriction of the co-simplicial object to $\Delta^{\leq n}$. Then we have the following tower
\[
	c\simeq \coBarP_\matheur{P}^0(c) \leftarrow \coBarP_\matheur{P}^1(c) \leftarrow \cdots \leftarrow \coBarP_{\matheur{P}}^n(c) \leftarrow \cdots
\]
and
\[
	\coBarP_\matheur{P}(c) \simeq \lim_n \coBarP_\matheur{P}^n(c).
\]

\begin{lem}\label{lem:vec_case_stability}
Let 
\[
	c\in \ComCoAlg(\Vect^{\leq -2}).
\]
Then, for all $n \geq 0$, the following natural map
\[
	\tr_{\geq -2^{n+1}+n+1} \coBarP_\matheur{P}^n(c) \to \tr_{\geq -2^{n+1}+n+1} \coBarP_\matheur{P}^{n-1}(c). \label{eq:vec_case_stability}
\]
is an equivalence.
\end{lem}
\begin{proof}
Let $F^n(c)$ denote the difference between $\coBarP^n_\matheur{P}(c)$ and $\coBarP^{n-1}_\matheur{P}(c)$, 
\[
	F^n(c) = \Fib(\coBarP^n_\matheur{P}(c) \to \coBarP^{n-1}_\matheur{P}(c)).
\]
Then for
\[
	c\in \ComCoAlg(\Vect^{\leq -2}),
\]
we see that
\[
	F^n(c) \in \Vect^{\leq -2\cdot 2^n + n} \simeq \Vect^{\leq -2^{n+1} +n}.
\]
Indeed, this is because of the fact that $c\in \Vect^{\leq -2}$ and hence, in the direct sum
\[
	\coBarP^\bullet_\matheur{P}(c)([n]) = \bigoplus_{m\geq 1} \matheur{P}^{\star n}(m) \otimes_{S_m} c^{\otimes m},
\]
$m=2^n$ is the first summand where we have non-degenerate ``(co-)cells.'' The shift to the right by $n$ is due to the fact that we are at level $n$ of the co-simplicial object.

As a consequence,
\[
	\tr_{\geq -2^{n+1}+n+1} \coBarP_\matheur{P}^n(c) \to \tr_{\geq -2^{n+1}+n+1} \coBarP_\matheur{P}^{n-1}(c)
\]
is an equivalence and we are done.
\end{proof}

\begin{cor} \label{cor:vec_case_from_coBar_to_coBarm}
Let
\[
	c\in \ComCoAlg(\Vect^{\leq -2}).
\]
Then, for any $n$, the following natural map
\[
	\tr_{\geq -n} \coBarP_\matheur{P}(c) \to \tr_{\geq -n} \coBarP_\matheur{P}^m(c)
\]
is an equivalence for all $m \gg 0$, where the bound depends only on $n$.
\end{cor}
\begin{proof}
The lemma follows from the general situation considered below. Suppose we have a sequence $X^0 \leftarrow X^1 \leftarrow \cdots$ and integers $n, m$ such that
\[
	F^i = \Fib(X^i \to X^m) \in \Vect^{\leq -n-2}
\]
for all $i\geq m$. Let $X = \lim X^i$ and note that
\[
	\Fib(X\to X^m) \simeq \Fib(\lim_{i \geq m} X^i \to X^m) \simeq \lim_{i\geq m} \Fib(X^i \to X^m) = \lim_{i\geq m} F^i.
\]

But now, the sequential limit can be computed as the fiber of two infinite products, i.e. we have the following fiber sequence
\[
	\lim_{i\geq m} F^i \to \prod F^i \to \prod F^i.
\]
Since the last two terms belong to $\Vect^{\leq -n-1}$, so is the first term. Therefore, 
\[
	\tr_{\geq -n} X \simeq \tr_{\geq -n} X_m
\]
and the proof concludes.
\end{proof}

\begin{rmk} \label{rmk:vec_case_bounded_cohomological_amplitude_inf_prod}
In the proof above, we use the fact that $\Vect^{\leq 0}$ is preserved under countable products in $\Vect$, or equivalently, that countable products are exact with respect to the usual $t$-structure on $\Vect$.  However, since the estimate appearing in~\eqref{eq:vec_case_stability} tends to $-\infty$, the conclusion of Corollary~\ref{cor:vec_case_from_coBar_to_coBarm} still holds true when countable products are only known to have uniformly bounded cohomological amplitude, i.e. there exists a fixed $N$ such that $\prod_i V_i$ lives in cohomological degrees $\leq N$ for any family $(V_i)_{i\in \mathbb{N}}$ such that $V_i$ lives in cohomological degrees $\leq 0$ for each $i$.
\end{rmk}

\subsubsection{} \label{subsubsec:proof_lem_Prim[-1]_is_good_Vec_case} We will now complete the proof of Lemma~\ref{lem:Prim[-1]_is_good_Vec_case}.

\begin{proof}[Proof of Lemma~\ref{lem:Prim[-1]_is_good_Vec_case}]	
The proof is now simple. In fact, we will only prove part~\eqref{lem:Prim[-1]_is_good_Vec_case:sifted_colims}, as the other one is almost identical. Note that due to~\eqref{eq:Prim[-1]_and_coBar}, what we prove about $\coBarP_\matheur{P}$ implies the corresponding statement of $\Prim[-1]$, up to a shift.

It suffices to show that for all $n$, we have
\[
	\tr_{\geq -n} \coBarP_\matheur{P} (\colim_\alpha c_\alpha) \simeq \tr_{\geq -n} \colim_{\alpha} \coBarP_\matheur{P} (c_\alpha)
\]
where $\alpha$ runs over some sifted diagram. But now, from Corollary~\ref{cor:vec_case_from_coBar_to_coBarm}, for all $m\gg 0$, we have
\begin{align*}
	\tr_{\geq -n} \coBarP_\matheur{P}(\colim_\alpha c_\alpha) 
	&\simeq \tr_{\geq -n} \coBarP_\matheur{P}^m (\colim_\alpha c_\alpha) 
	\simeq \tr_{\geq -n} \colim_\alpha \coBarP_\matheur{P}^m(c_\alpha)	
	\simeq \colim_\alpha \tr_{\geq -n} \coBarP_\matheur{P}^m(c_\alpha) \\
	&\simeq \colim_\alpha \tr_{\geq -n} \coBarP_\matheur{P}(c_\alpha) 
	\simeq \tr_{\geq -n} \colim_\alpha \coBarP_\matheur{P}(c_\alpha).
\end{align*}

\end{proof}

\begin{rmk}
The cohomological estimate done above implies that
\[
	\coBarP_{\ComCoAlg}(c) \in \Lie[-1](\Vect^{\leq -2}),
\]
or equivalently, that
\[
	\Prim[-1](c) \in \Lie(\Vect^{\leq -1}),
\]
when
\[
	c\in \ComCoAlg(\Vect^{\leq -2}).
\]

Indeed, from Corollary~\ref{cor:vec_case_from_coBar_to_coBarm}, we know that for some $m\gg 0$, 
\[
	\tr_{\geq -1} \coBarP_{\matheur{P}}(c) \simeq \tr_{\geq -1} \coBarP_{\matheur{P}}^m(c),
\]
and moreover, a downward induction using Lemma~\ref{lem:vec_case_stability} shows that
\[
	\tr_{\geq -1} \coBarP_{\matheur{P}}^m(c) \simeq \tr_{\geq -1} \coBarP^0_\matheur{P}(c) \simeq \tr_{\geq -1} c \simeq 0.
\]
\end{rmk}

\subsubsection{} The following result concludes the proof of Theorem~\ref{thm:equivalence_koszul_vec_case}.

\begin{lem} \label{lem:vec_case_Prim[-1]_conservative}
	The functor
	\[
		\Prim[-1]: \ComCoAlg(\Vect^{\leq -2}) \to \Lie(\Vect^{\leq -1})
	\]
	is conservative.
\end{lem}
\begin{proof}
It suffices to show that
\[
	\coBarP_\matheur{P}: \ComCoAlg(\Vect^{\leq -2}) \to \Lie[-1](\Vect^{\leq -2})
\]
is conservative, and we will prove that by contradiction. Namely, let
\[
	f: c_1\to c_2
\]
be a morphism in $\ComCoAlg(\Vect^{\leq -2})$ such that $f$ is not an equivalence. Suppose that
\[
	\coBarP_\matheur{P}(f): \coBarP_\matheur{P}(c_1) \to \coBarP_\matheur{P}(c_2)
\]
is an equivalence, we will derive a contradiction.

Let $k$ be the smallest number such that 
\[
	\tr_{\geq -k}(f): \tr_{\geq -k} c_1 \to \tr_{\geq -k} c_2
\]
is not an equivalence. Now, by Corollary~\ref{cor:vec_case_from_coBar_to_coBarm}, we know that there is some $m\gg 0$ such that
\[
	\tr_{\geq -k} \coBarP_\matheur{P}(c_i)  \simeq \tr_{\geq -k} \coBarP_\matheur{P}^m(c_i)
\]
for $i\in \{1, 2\}$. Thus, we know that
\[
	\tr_{\geq -k} \coBarP_\matheur{P}^m(c_1) \to \tr_{\geq -k} \coBarP_\matheur{P}^m(c_2)
\]
is an equivalence.

By an estimate similar to the one at Lemma~\ref{lem:vec_case_stability}, we will show that 
\[
	\tr_{\geq -k} F^n(c_1) \simeq \tr_{\geq -k} F^n(c_2)
\]
for all $n \geq 1$, where $F^*(-)$ denotes the fiber as in the proof of Lemma~\ref{lem:vec_case_stability}.  Indeed, the difference between $F^n(c_1)$ and $F^n(c_2)$ lies in cohomological degrees
\[
	\leq -2(2^n-1) - k + n = -2^{n+1} - k + n + 2 < -k, \quad \forall n \geq 1.
\]

And hence, a downward induction, starting from $n=m$, using the diagram
\[
\xymatrix{
	F^n(c_1) \ar[d] \ar[r] & \coBarP_\matheur{P}^n(c_1) \ar[d] \ar[r] & \coBarP_\matheur{P}^{n-1}(c_1) \ar[d] \\
	F^n(c_2) \ar[r] & \coBarP_\matheur{P}^n(c_2) \ar[r] & \coBarP_\matheur{P}^{n-1}(c_2) 
}
\]
implies that
\[
	\tau_{\geq -k} c_1 \simeq \tau_{\geq -k} c_2,
\]
which contradicts our original assumption. Hence, we are done.
\end{proof}

\begin{rmk}\label{rmk:general_category}
	Note that the proof we gave above could be carried out in a more general setting. Namely, the only properties of $\Vect$ that we used are
	\begin{enumerate}[(i)]
		\item The symmetric monoidal structure is right exact (namely, it preserved $\Vect^{\leq 0}$).
		\item The $t$-structure on $\Vect$ is left separated.
		\item Countable products have uniformly bounded cohomological amplitude (see Remark~\ref{rmk:vec_case_bounded_cohomological_amplitude_inf_prod}).
	\end{enumerate}
\end{rmk}
\begin{rmk}\label{rmk:general_operads}
	We can also replace the operad $\Lie$ by any operad $\matheur{O}$ such that 
	\begin{enumerate}[(i)]
		\item $\matheur{O}$ is classical, i.e. it lies in the heart of the $t$-structure of $\Vect$.
		\item $\matheur{O}^\vee[-1]$ is also classical.
		\item $\matheur{O}(1) \simeq \Lambda$ (as we already assume throughout this paper).
	\end{enumerate}
\end{rmk}

\subsection{Higher enveloping algebras}
We will briefly explain the topological analogue of the main results in the factorizable setting, proved in the next subsection. In this setting, the result is an immediate consequence of what we already proved above.

The main reference of this part is~\cite{gaitsgory_study_2017}.

\subsubsection{} Let
\[
	\mathfrak{g} \in \Lie(\Vect).
\]
Then one can form its $E_n$-universal enveloping algebra
\[
	U_{E_n}(\mathfrak{g}) \in E_n(\Vect)
\]
by applying the following sequence of functors
\[
\xymatrix{
	\Lie(\Vect) \ar[rr]^<<<<<<<<<<{\Omega^{\times n} \simeq [-n]} && E_n(\Lie(\Vect)) \ar[rr]^<<<<<<<<<<{E_n(\Chev)} && E_n(\ComCoAlg(\Vect)) \ar[rr]^<<<<<<<<<<{\oblv_{\ComCoAlg}} && E_n(\Vect)
}
\]
where $E_n(\Lie(\Vect))$ and $E_n(\ComCoAlg(\Vect))$ are categories of $E_n$-algebras with respect to the Cartesian monoidal structure on $\Lie(\Vect)$ and $\ComCoAlg(\Vect)$ respectively (note that the latter one is just the given by $\otimes$ in $\Vect$).

\subsubsection{} It is proved in~\cite{gaitsgory_study_2017} that $[-n]$ induces an equivalence
\[
	[-n]: \Lie(\Vect) \simeq E_n(\Lie(\Vect)) : [n].
\]
Moreover, we know from Theorem~\ref{thm:equivalence_koszul_vec_case} that
\[
	E_n(\Chev): E_n(\Lie(\Vect^{\leq -1})) \to E_n(\ComCoAlg(\Vect^{\leq -2})).
\]
As a result, we get
\begin{prop}
We have the following equivalence of categories
\[
	\Lie(\Vect^{\leq -n-1}) \simeq E_n(\ComCoAlg(\Vect^{\leq -2})). \teq \label{eq:higher_enveloping_alg_equivalence}
\]
\end{prop}

\subsubsection{} The equivalence~\eqref{eq:higher_enveloping_alg_equivalence} is precisely what we are looking for in the context of factorization algebras on the Ran space in the following subsection. One part of the work is to find connectivity assumptions on $\Shv(\Ran X)$ which mirror those appearing in $\Vect^{\leq -n-1}$ and $\Vect^{\leq -2}$ respectively.

\subsection{The case of $\Liestar$- and $\ComCoAlgstar$-algebras on $\Ran X$} \label{subsec:koszul_equivalence_Lie_ComCoAlg_Ran_case}
We now come to the precise formulation and the proof of Theorem~\ref{thm:intro:Koszul_duality_connectivity_on_Ran}. Throughout this subsection, we will assume that $X$ is smooth over $k$ of dimension $d$.

\begin{defn} \label{defn:connectivity_constraints_Lie_ComCoAlg_Ran}
	Let $\Shv(\Ran X)^{\leq c_{cA}}$ and $\Shv(\Ran X)^{\leq c_L}$ denote the full subcategory of $\Shv(\Ran X)$ consisting of sheaves $\matheur{F}$ such that for all non-empty finite sets $I$,
	\[
		\matheur{F}|_{\oversetsupscript{X}{\circ}{I}} \in \Shv(\oversetsupscript{X}{\circ}{I})^{\leq (-1-d)|I| - 1},
	\]
	and respectively,
	\[
		\matheur{F}|_{\oversetsupscript{X}{\circ}{I}} \in \Shv(\oversetsupscript{X}{\circ}{I})^{\leq (-1-d)|I|}.
	\]
	Here, we use the perverse $t$-structure, and $X$ is a scheme of pure dimension $d$.
\end{defn}

\begin{notation}
	We will use
	\[
		\Liestar(\Ran X)^{\leq c_L} \qquad\text{and}\qquad \ComCoAlgstar(\Ran X)^{\leq c_{cA}}
	\]
	to denote
	\[
		\Liestar(\Shv(\Ran X)^{\leq c_L}) \qquad\text{and}\qquad \ComCoAlgstar(\Ran X)^{\leq c_{cA}}
	\]
	respectively.
\end{notation}

With these connectivity assumptions in mind, the rest of this subsection will be devoted to the proof of the following
\begin{thm} \label{thm:Koszul_duality_connectivity_on_Ran}
Suppose $X$ is smooth over $k$ of dimension $d$. We have the following commutative diagram
\[
\xymatrix{
	\Liestar(\Ran X)^{\leq c_L} \ar@{=}[rr]^<<<<<<<<<<{\Chev}_<<<<<<<<<<{\Prim[-1]} && \ComCoAlgstar(\Ran X)^{\leq c_{cA}} \\
	\Liestar(X)^{\leq c_L} \ar@{^(->}[u] \ar@{=}[rr]^{\Chev}_{\Prim[-1]} && \coFactstar(X)^{\leq c_{cA}} \ar@{^(->}[u]
} \teq \label{eq:diagram_in_thm:Koszul_duality_connectivity_on_Ran}
\]
where $\leq c_L$ and $\leq c_{cA}$ denote the connectivity constraints given in Definition~\ref{defn:connectivity_constraints_Lie_ComCoAlg_Ran}, and where $\Chev$ and $\Prim[-1]$ are the functors coming from Koszul duality.
\end{thm}

\begin{rmk} \label{rmk:Ran_case_supress_conn_cond}
As in the case of $\Vect$, we will in general suppress the distinction between a functor and its restriction to a subcategory cut out by some connectivity condition. For example, we will write $\Chev$ instead of $\Chev|_{\Liestar(\Ran X)^{\leq c_L}}$ unless confusion is likely to occur.
\end{rmk}

\begin{rmk}
As in Remark~\ref{rmk:vec_case_land_in_correct_category}, it is straightforward to see that $\Chev$ restricts to the correct subcategories. For $\Prim[-1]$, it is a direct consequence of Lemma~\ref{lem:ran_case_stability} and Corollary~\ref{cor:Ran_case_from_coBar_to_coBarm}.
\end{rmk}

\begin{rmk}
As in the case of $\Vect$, the operad/co-operad pair $\Lie$ and $\ComCoAlg$ could be replaced by a pair of Koszul dual operad/co-operad $\matheur{O}$ and $\matheur{O}^\vee$ satisfying the conditions listed in Remark~\ref{rmk:general_operads}.\footnote{Note that for a general operad $\matheur{O}$, only the first row of~\eqref{eq:diagram_in_thm:Koszul_duality_connectivity_on_Ran} makes sense.}
\end{rmk}

We start with a preliminary lemma, which ensures that the categories
\[
	\Liestar(\Ran X)^{\leq c_L} \qquad\text{and}\qquad \ComCoAlgstar(\Ran X)^{\leq c_{cA}}
\]
are actually well-defined.

\begin{lem} \label{lem:otimesstar_preserves_cL_ccA}
Suppose $X$ is smooth over $k$ of dimension $d$. Then the subcategories $\Shv(\Ran X)^{\leq c_L}$ and $\Shv(\Ran X)^{\leq c_{cA}}$ are preserved under the $\otimesstar$-monoidal structure on $\Shv(\Ran X)$.
\end{lem}

\begin{proof}
Recall from~\eqref{eq:otimesstar_explicit_formula} that if
\[
	\matheur{F}_1, \dots, \matheur{F}_k \in \Shv(\Ran X),
\]
then from the definition of $\otimesstar$, we have
\[
	(\matheur{F}_1 \otimesstar \cdots \otimesstar \matheur{F}_k)|_{\oversetsupscript{X}{\circ}{I}} \simeq \bigoplus_{I = \cup_{i=1}^k I_i} \Delta^!_{\sqcup_{i=1}^k I_i \surjects \cup_{i=1}^k I_i} (\matheur{F}_1 \boxtimes\cdots\boxtimes \matheur{F}_k)|_{\oversetsupscript{X}{\circ}{I}}. \teq \label{eq:otimesstar_as_a_sum}
\]
	
Now, suppose that
\[
	\matheur{F}_1, \dots, \matheur{F}_k \in \Shv(\Ran X)^{\leq c_L},
\]
then we see that each summand in~\eqref{eq:otimesstar_as_a_sum} lies in perverse cohomological degrees
\begin{align*}
	&\leq  (-1-d)\sum_{i=1}^k |I_i| + d\left(\sum_{i=1}^k |I_i| - |I|\right) \\
	&\leq -\sum_{i=1}^k |I_i| - d|I|  \\
	&\leq (-1-d)|I|.
\end{align*}
Here, the first inequality is due to the fact that the map
\[
	\oversetsupscript{X}{\circ}{I} \to \prod_{i=1}^k \oversetsupscript{X}{\circ}{I_i}
\]
is a regular embedding (since $X$ is smooth), and that the (perverse) cohomological amplitude of the $!$-pullback along a regular embedding is equal to the codimension. The sequence of inequalities above thus implies that
\[
	\matheur{F}_1\otimesstar\cdots\otimesstar\matheur{F}_k \in \Shv(\Ran X)^{\leq c_L}.
\]

Similarly, suppose that
\[
	\matheur{F}_1, \dots, \matheur{F}_k \in \Shv(\Ran X)^{\leq c_{cA}},
\]
then each summand in~\eqref{eq:otimesstar_as_a_sum} lies in perverse cohomological degrees
\begin{align*}
	&\leq (-1-d)\sum_{i=1}^k |I_i| - k + d\left(\sum_{i=1}^k |I_i| - |I|\right) \teq \label{eq:cohomological_estimates_cA_otimesstar} \\
	&\leq -\sum_{i=1}^k |I_i| - k - d|I| \\
	&\leq (-1-d)|I| - 1.
\end{align*}
Thus,
\[
	\matheur{F}_1 \otimesstar \cdots \otimesstar \matheur{F}_k \in \Shv(\Ran X)^{\leq c_{cA}},
\]
which concludes the proof.
\end{proof}

\subsubsection{} Back to Theorem~\ref{thm:Koszul_duality_connectivity_on_Ran}. First, we will prove the equivalence on the top row of~\eqref{eq:diagram_in_thm:Koszul_duality_connectivity_on_Ran}. Then, we will show that it induces an equivalence between the corresponding sub-categories on the bottom row.

As in the case of $\Vect$, to prove that $\Chev$ and $\Prim[-1]$ are mutually inverse functors, it suffices to show that $\Chev$ is fully-faithful, and $\Prim[-1]$ is conservative. As above, we start with the following lemma, whose proof, after some preparation, will conclude in \S\ref{subsubsec:proof_lem_Prim[-1]_is_good_Ran_case}.

\begin{lem} \label{lem:Prim[-1]_is_good_Ran_case}
	The functor $\Prim[-1]|_{\ComCoAlgstar(\Ran X)^{\leq c_{cA}}}$ satisfies the following conditions (see Remark~\ref{rmk:Ran_case_supress_conn_cond})
	\begin{enumerate}[(i)]
		\item $\Prim[-1]$ commutes with sifted colimits.
		\item The natural map
		\[
			\Free_{\Lie} \to \Prim[-1]\circ \triv_{\ComCoAlg}
		\]
		is an equivalence.
	\end{enumerate}
\end{lem}

As in Corollary~\ref{cor:vect_case_Chev_fully_faithful}, this immediately implies the following
\begin{cor} \label{cor:Ran_case_Chev_fully_faithful}
	$\Chev|_{\Liestar(\Ran X)^{\leq c_L}}$ is fully faithful.
\end{cor}

\subsubsection{} In essence, the strategy we follow here is identical to that of the $\Vect$ case even though the actual execution might seem somewhat more involved. The main observation (which is new compared to the case of $\Vect$) is that to prove the equivalences involved in Lemma~\ref{lem:Prim[-1]_is_good_Ran_case}, it suffices to prove them after after pulling back to $\oversetsupscript{X}{\circ}{I}$ for each non-empty finite set $I$.

\subsubsection{} In general, for any
\[
	\matheur{A} \in \ComCoAlgstar(\Ran X)^{\leq c_{cA}},
\]
we have
\[
	\coBarP_{\ComCoAlg}(\matheur{A}) = \Tot(\coBarP^\bullet_{\ComCoAlg}(\matheur{A})),
\]
where $\coBarP^\bullet_{\ComCoAlg}(\matheur{A})$ is a co-simplicial object.

Let
\[
	\coBarP^n_{\ComCoAlg}(\matheur{A}) = \Tot(\coBarP^\bullet_{\ComCoAlg}(\matheur{A})|_{\Delta^{\leq n}}).
\]
Then, we have the following tower
\[
	\matheur{A} \simeq \coBarP^0_{\ComCoAlg}(\matheur{A}) \leftarrow \coBarP^1_{\ComCoAlg}(\matheur{A}) \leftarrow \cdots 
\]
and
\[
	\coBarP_{\ComCoAlg}(\matheur{A}) \simeq \lim_n \coBarP^n_{\ComCoAlg}(\matheur{A}).
\]

\subsubsection{} Let
\[
	F^n(\matheur{A}) = \Fib(\coBarP^n_{\ComCoAlg}(\matheur{A}) \to \coBarP^{n-1}_{\ComCoAlg}(\matheur{A})),
\]
and $I$ a non-empty finite set. Using the same argument as in the case of $\Vect$ in combination with the cohomological estimate~\eqref{eq:cohomological_estimates_cA_otimesstar}, we see that $F^n(\matheur{A})|_{\oversetsupscript{X}{\circ}{I}}$ lives in cohomological degrees
\begin{align*}
	&\leq (-1-d)\sum_{i=1}^{2^n} |I_i| - 2^n + d\left(\sum_{i=1}^{2^n} |I_i| -|I|\right) + n \\
	&= -\sum_{i=1}^{2^n} |I_i| - 2^n - d|I| + n \\
	&\leq -2^{n+1} - d|I| + n
\end{align*}
which goes to $-\infty$ when $n \to \infty$.

This gives us the following analog of Lemma~\ref{lem:vec_case_stability}.
\begin{lem} \label{lem:ran_case_stability}
Let
\[
	\matheur{A} \in \ComCoAlgstar(\Ran X)^{\leq c_{cA}}.
\]
Then, for any $n$ and $I$, the following natural map
\[
	\tr_{\geq -2^{n+1} - d|I| + n + 1} (\coBarP^n_{\ComCoAlg}(\matheur{A})|_{\oversetsupscript{X}{\circ}{I}}) \to \tr_{\geq -2^{n+1} - d|I| + n + 1} (\coBarP^{n-1}_{\ComCoAlg}(\matheur{A})|_{\oversetsupscript{X}{\circ}{I}})
\]
is an equivalence.
\end{lem}

This implies the following result, which is parallel to Corollary~\ref{cor:vec_case_from_coBar_to_coBarm}. See also Remark~\ref{rmk:vec_case_bounded_cohomological_amplitude_inf_prod}, \cite[Lemma~3.2.1]{liu_enhanced_2014} and the discussion after it where left-completeness and uniformly bounded cohomological amplitude for countable products are discussed.
\begin{cor} \label{cor:Ran_case_from_coBar_to_coBarm}
Let
\[
	\matheur{A} \in \ComCoAlgstar(\Ran X)^{\leq c_{cA}}.
\]
Then, for any $n$ and $I$, the following natural map
\[
	\tr_{\geq -n} (\coBarP_{\ComCoAlg}(\matheur{A})|_{\oversetsupscript{X}{\circ}{I}}) \to \tr_{\geq -n} (\coBarP_{\ComCoAlg}^m(\matheur{A})|_{\oversetsupscript{X}{\circ}{I}})
\]
is an equivalence, when $m \gg 0$ depending only on $n$ and $I$.
\end{cor}

\subsubsection{Concluding the proof of Lemma~\ref{lem:Prim[-1]_is_good_Ran_case}} \label{subsubsec:proof_lem_Prim[-1]_is_good_Ran_case}
As in the proof of Lemma~\ref{lem:Prim[-1]_is_good_Vec_case}, Lemma~\ref{lem:Prim[-1]_is_good_Ran_case} is now a direct consequence of Lemma~\ref{lem:ran_case_stability} and Corollary~\ref{cor:Ran_case_from_coBar_to_coBarm}.

\begin{rmk}
	Note that when $X$ is a point, namely when $d = \dim X = 0$, the cohomological estimates in Lemma~\ref{lem:ran_case_stability} recover those of Lemma~\ref{lem:vec_case_stability}. 
\end{rmk}

To finish with the top equivalence in~\eqref{eq:diagram_in_thm:Koszul_duality_connectivity_on_Ran}, we need the following
\begin{lem} \label{lem:Ran_case_Prim[-1]_conservative}
The functor
\[
	\Prim[-1]: \ComCoAlgstar(\Ran X)^{\leq c_{cA}} \to \Liestar(\Ran X)^{\leq c_L}
\]
is conservative.
\end{lem}
\begin{proof}
It suffices to show that
\[
	\coBarP_{\ComCoAlg}: \ComCoAlgstar(\Ran X)^{\leq c_{cA}} \to \Liestar[-1](\Ran X)^{\leq c_{cA}}
\]
is conservative, and we will do so by contradiction. Namely, let
\[
	f: \matheur{A}_1 \to \matheur{A}_2
\]
be a morphism in $\ComCoAlgstar(\Ran X)^{\leq c_{cA}}$ that is not an equivalence. Suppose that
\[
	\coBarP_{\ComCoAlg}(f): \coBarP_{\ComCoAlg}(\matheur{A}_1) \to \coBarP_{\ComCoAlg}(\matheur{A}_2)
\]
is an equivalence, we will derive a contradiction.

Let $I$ the set of smallest cardinality	 such that the map
\[
	f|_{\oversetsupscript{X}{\circ}{I}}: \matheur{A}_1|_{\oversetsupscript{X}{\circ}{I}} \to \matheur{A}_2|_{\oversetsupscript{X}{\circ}{I}}
\]
is not an equivalence. Let $k \geq 0$  be the smallest number such that
\[
	\tr_{\geq (-1-d)|I|-1-k}(\matheur{A}_1|_{\oversetsupscript{X}{\circ}{I}})  \to \tr_{\geq (-1-d)|I|-1-k} (\matheur{A}_2|_{\oversetsupscript{X}{\circ}{I}})
\]
is not an equivalence.

By Corollary~\ref{cor:Ran_case_from_coBar_to_coBarm}, we know that there exists some $m\gg 0$ such that
\[
	\tr_{\geq (-1-d)|I|-1-k} (\coBarP_{\ComCoAlg}(\matheur{A}_i)|_{\oversetsupscript{X}{\circ}{I}}) \simeq \tr_{\geq (-1-d)|I|-1-k} (\coBarP^m_{\ComCoAlg}(\matheur{A}_i)|_{\oversetsupscript{X}{\circ}{I}})
\]
for $i \in \{1, 2\}$. Thus, we get the following equivalence
\[
	\tr_{\geq (-1-d)|I|-1-k}(\coBarP^m_{\ComCoAlg}(\matheur{A}_1)|_{\oversetsupscript{X}{\circ}{I}}) \simeq \tr_{\geq (-1-d)|I|-1-k} (\coBarP^m_{\ComCoAlg}(\matheur{A}_2)|_{\oversetsupscript{X}{\circ}{I}}).
\]
But observe that if we let
\[
	F^n(\matheur{A}_i) = \Fib(\coBarP_{\ComCoAlg}^n(\matheur{A}_i) \to \coBarP_{\ComCoAlg}^{n-1}(\matheur{A}_i))
\]
then the difference between $F^n(\matheur{A}_1)|_{\oversetsupscript{X}{\circ}{I}}$ and $F^n(\matheur{A}_2)|_{\oversetsupscript{X}{\circ}{I}}$ lies in cohomological degrees
\begin{align*}
	&\leq (-1-d)|I| - 1 - k + (-1-d)\sum_{i=1}^{2^n-1} |I_i| - (2^n - 1) + d\left(|I| + \sum_{i=1}^{2^n - 1}|I_i| - |I|\right) + n \\
	&\leq (-1-d)|I| - 1 -k - \sum_{i=1}^{2^n-1} |I_i| - 2^n +1 + n \\
	&< (-1-d)|I| - 1 - k.
\end{align*}
This implies that for $n\geq 1$,
\[
	\tr_{\geq (-1-d)|I|-1-k} (F^n(\matheur{A}_1)|_{\oversetsupscript{X}{\circ}{I}}) \simeq \tr_{\geq (-1-d)|I|-1-k} (F^n(\matheur{A}_2)|_{\oversetsupscript{X}{\circ}{I}}).
\]
Thus, as in the case of $\Vect$, a downward induction implies that
\[
	\tr_{\geq (-1-d)|I|-1-k}(\matheur{A}_1|_{\oversetsupscript{X}{\circ}{I}}) \simeq \tr_{\geq (-1-d)|I|-1-k}(\matheur{A}_2|_{\oversetsupscript{X}{\circ}{I}}),
\]
which contradicts our original assumption, and we are done.
\end{proof}

\subsubsection{} Corollary~\ref{cor:Ran_case_Chev_fully_faithful} and Lemma~\ref{lem:Ran_case_Prim[-1]_conservative} together prove the equivalence on the top row of diagram~\eqref{eq:diagram_in_thm:Koszul_duality_connectivity_on_Ran}. It remains to show the equivalence in the bottom row, for which it suffices to show that for any
\[
	\mathfrak{g} \in \Liestar(\Ran X)^{\leq c_L},
\]
$\Chev(\mathfrak{g})$ is factorizable if and only if $\mathfrak{g} \in \Liestar(X)^{\leq c_L}$.

\subsubsection{}
For the ``if'' direction, recall that as a consequence of~\cite[Thm. 6.4.2 and 5.2.1]{francis_chiral_2011}, we know that the functor
\[
	\Chev: \Liestar(X) \to \ComCoAlgstar(\Ran X)
\]
lands inside the full-subcategory $\coFactstar(X)$ of factorizable co-algebras. We thus get a functor
\[
	\Chev: \Liestar(X)^{\leq c_L} \to \coFactstar(X)^{\leq c_{cA}},
\]
which settles the ``if'' direction.

\subsubsection{}
For the ``only if'' direction, let
\[
	\mathfrak{g} \in \Liestar(\Ran X)^{\leq c_L}
\]
whose support does not lie in $X$. We will show that $\Chev\mathfrak{g}$ is not factorizable.

Using the $\assgr\circ \addFil$ trick (see \S\ref{sec:appendix:addFil_trick}), it suffices to prove for the case where $\mathfrak{g}$ is a trivial (i.e. abelian) Lie algebra. In that case, we know that
\[
	\Chev \mathfrak{g} = \Sym^{>0}(\mathfrak{g}[1]),
\]
where $\Sym$ is taken using the $\otimesstar$-monoidal structure.

Let $I$ be the smallest set, with $|I| > 1$, such that $\mathfrak{g}|_{\oversetsupscript{X}{\circ}{I}} \not\simeq 0$. Now, it's easy to see that $\Sym^{>0}(\mathfrak{g}[1])$ fails the factorizability condition at $\oversetsupscript{X}{\circ}{I}$, which concludes the ``only if'' direction.

\section{Factorizability of $\coChev$}

In this section, we will prove Theorem~\ref{thm:intro:factorizability_coChev}, which asserts that when $\mathfrak{g} \in \coLiestar(X)$ satisfies a certain co-connectivity constraint, the commutative algebra
\[
	\coChev(\mathfrak{g}) \in \ComAlgstar(\Ran X)
\]
is factorizable.

Note that an analog of this result, where $\coChev$ is replaced by $\Chev$, has been proved in~\cite{francis_chiral_2011} (and in fact, we used this result in the previous section). The main difficulties of the $\coChev$ case stem from the fact that, unlike $\Chev$, $\coChev$ is defined as a limit, and most of the functors that we want it to interact with don't generally commute with limits.

As above, our main strategy is to introduce a certain co-connectivity condition to ensure that when one takes the limit of a diagram involving objects satisfying it, the answer, in some sense, converges instead of running off to infinity, so we still have a good control over it. 

We start with the precise statement of the theorem. Then, after a quick digression on the various notions related to the convergence of a limit, we will present the main strategy. Finally, the proof itself will be given.

\subsection{The statement}
We start with the co-connectivity conditions.

\begin{defn} \label{defn:connectivity_constraints_coLie_ComAlg_Ran}
	Let $\Shv(\Ran X)^{\geq n}$ denote the full subcategory of $\Shv(\Ran X)$ consisting of sheaves $\matheur{F}$ such that for all non-empty finite sets $I$,
	\[
		\matheur{F}|_{\oversetsupscript{X}{\circ}{I}} \in \Shv(\oversetsupscript{X}{\circ}{I})^{\geq n},
	\]
	As before, we use the perverse $t$-structure.
\end{defn}

\begin{notation}
We will use
\[
	\coLiestar(\Ran X)^{\geq n} \qquad\text{and}\qquad \ComAlgstar(\Ran X)^{\geq n}
\]
to denote
\[
	\coLiestar(\Shv(\Ran X)^{\geq n}) \qquad\text{and}\qquad \ComAlgstar(\Shv(\Ran X)^{\geq n})
\]
respectively.
\end{notation}

Our main goal is to prove the following
\begin{thm} \label{thm:factorizability_coChev}
Restricted to the full subcategory $\coLiestar(X)^{\geq 1}$ of $\coLiestar(\Ran X)^{\geq 1}$ consisting of $\coLie$-coalgebras whose underlying sheaves are supported on the diagonal $X$, the functor $\coChev$ factors through $\Factstar$, i.e. we have the following commutative diagram
\[
\xymatrix{
	\coLiestar(X)^{\geq 1} \ar[dr]_{\coChev} \ar[rr]^{\coChev} && \ComAlgstar(\Ran X) \\
	& \Factstar(X) \ar@{^(->}[ur]
}
\]
In other words, $\coChev \mathfrak{g}$ is factorizable when $\mathfrak{g} \in \coLiestar(X)^{\geq 1}$.
\end{thm}

\subsection{Stabilizing co-filtrations and decaying sequences (a digression)} We will now describe a condition on co-filtered and graded objects which make them behave nicely with respect to taking limits. 

\begin{defn}
Let $\matheur{C}$ be a stable infinity category equipped with a $t$-structure. Then, a co-filtered object $c\in \matheur{C}^{\coFil^{>0}}$ (see \S\ref{sec:appendix:cofiltration_addCoFil}) is said to stabilize if for all $n$, the induced map
\[
	\tr_{\leq n} c_m \rightarrow \tr_{\leq n} c_{m+1}
\]
is an equivalence for all $m\gg 0$.

A graded object $c\in \matheur{C}^{\gr^{>0}}$ is said to be decaying if for all $n$, we have
\[
	\tr_{\leq n} c_m \simeq 0
\]
for all $m\gg 0$.
\end{defn}

\begin{notation}
We use $\matheur{C}^{\coFil^{>0}, \stab}$ and $\matheur{C}^{\gr^{>0}, \decay}$ to denote the subcategories of $\matheur{C}^{\coFil^{>0}}$ and $\matheur{C}^{\gr^{>0}}$ consisting of stabilizing and decaying objects respectively.
\end{notation}

We have the following lemmas, whose proofs are straightforward.
\begin{lem}
Let $c\in \matheur{C}^{\coFil^{>0}}$. Then $c\in \matheur{C}^{\coFil^{>0} , \stab}$ if and only if $\assgr c \in \matheur{C}^{\gr^{>0}, \decay}$.
\end{lem}

\begin{lem} \label{lem:limit_of_stabilizing_cofiltration}
If $c \in \matheur{C}^{\coFil^{>0}, \stab}$, then for each $n$, the natural map
\[
	\tau_{\leq n} \oblv_{\coFil} c \to \tau_{\leq n} c_m
\]
is an equivalence when $m\gg 0$.
\end{lem}
\begin{proof}
	By throwing away finitely many terms at the beginning, without loss of generality, we can assume that the natural maps
	\[
		\tau_{\leq n+1} c_i \to \tau_{\leq n+1} c_j, \qquad \forall i\geq j>0
	\]
	are all equivalences. Now, it suffices to show that the following map is an equivalence
	\[
		\tau_{\leq n} \lim_i c_i \to \tau_{\leq n} c_1.
	\]
	Equivalently, it suffices to show that
	\[
		\Fib(\lim_i c_i \to c_1) \in \matheur{C}^{\geq n+1}.
	\]
	
	However,
	\[
		\Fib(\lim_i c_i \to c_1) \simeq \lim_i(\Fib(c_i \to c_1)) \in \matheur{C}^{\geq n+1}
	\]
	because
	\[
		\Fib(c_i \to c_1) \in \matheur{C}^{\geq n+1}, \qquad\forall i.
	\]
	Hence, we are done, since
	\[
		i_{\geq n+1}: \matheur{C}^{\geq n+1} \to \matheur{C}
	\]
	commutes with limits (see \S\ref{subsubsec:t-structure_convention}).
\end{proof}

\begin{lem}
The natural transformation
\[
	\bigoplus \to \prod
\]
between functors
\[
	\matheur{C}^{\gr^{>0}, \decay} \to \matheur{C}
\]
is an equivalence.
\end{lem}
\begin{proof}
Note that
\[
	\prod_i c_i \simeq \lim_k \bigoplus_{i\leq k} c_i.
\]
Moreover, since the sequence we are taking the limit over stabilizes, the result follows as a direct consequence of Lemma~\ref{lem:limit_of_stabilizing_cofiltration}.
\end{proof}

\subsubsection{} The various definitions and observations above have straightforward analogues in the case of sheaves on the $\Ran$ space.
\begin{defn}
A co-filtered sheaf $\matheur{F} \in \Shv(\Ran X)^{\coFil^{>0}}$ is said to stabilize if for any non-empty finite set $I$,
\[
	\matheur{F}|_{\oversetsupscript{X}{\circ}{I}} \in \Shv(\oversetsupscript{X}{\circ}{I})^{\coFil^{>0}, \stab}.
\]

Similarly, a graded sheaf $\matheur{F} \in \Shv(\Ran X)^{\gr^{>0}}$ is said to be decaying if for any non-empty finite set $I$,
\[
	\matheur{F}|_{\oversetsupscript{X}{\circ}{I}} \in \Shv(\oversetsupscript{X}{\circ}{I})^{\gr^{>0}, \decay}.
\]
\end{defn}

\begin{notation}
	We use $\Shv(\Ran X)^{\coFil^{>0}, \stab}$ and $\Shv(\Ran X)^{\gr^{>0}, \decay}$ to denote the full-subcategories of $\Shv(\Ran X)^{\coFil^{>0}}$ and $\Shv(\Ran X)^{\gr^{>0}}$ consisting of stabilizing and decaying objects, respectively.
\end{notation}

It's straightforward to see that the following analogs of the lemmas above still hold in this setting.

\begin{lem} \label{lem:stab_decay_equivalence_Ran_case}
Let $\matheur{F} \in \Shv(\Ran X)^{\coFil^{>0}}$. Then $\matheur{F} \in \Shv(\Ran X)^{\coFil^{>0}, \stab}$ if and only if $\assgr\matheur{F} \in \Shv(\Ran X)^{\gr^{>0}, \decay}$.
\end{lem}

\begin{lem}
	If $\matheur{F} \in \Shv(\Ran X)^{\coFil^{>0}, \stab}$, then for each $I$ and $n$, the natural map\footnote{Note that $\oblv_{\coFil}$ commutes with restricting to $\oversetsupscript{X}{\circ}{I}$ for any non-empty, finite set $I$. Thus, the LHS is free of ambiguity.}
	\[
		\tau_{\leq n} \oblv_{\coFil} \matheur{F}|_{\oversetsupscript{X}{\circ}{I}} \to \tau_{\leq n} \matheur{F}_m|_{\oversetsupscript{X}{\circ}{I}}
	\]
	is an equivalence when $m\gg 0$.
\end{lem}

\begin{lem} \label{lem:direct_sum_and_products_are_the_same_Ran}
The natural transformation
\[
	\bigoplus \to \prod
\]
between functors
\[
	\Shv(\Ran X)^{\gr^{>0}, \decay} \to \Shv(\Ran X)
\]
is an equivalence.
\end{lem}

\subsection{Strategy} To prove that $\Chev \mathfrak{g}$ is factorizable when $\mathfrak{g}\in \Liestar(X)$, \cite{francis_chiral_2011} uses the $\addFil$ trick (see \S\ref{sec:appendix:addFil_trick}) to reduce to the case where $\mathfrak{g}$ is a trivial. When $\mathfrak{g}$ is trivial, we have
\[
\Chev\mathfrak{g} \simeq \Sym^{>0}\mathfrak{g},
\]
and the result can be seen directly.

In the case of $\coChev$, while the core strategy remains the same, it is more complicated to carry out since many commutative diagrams needed for the $\addFil$ trick to work (see~\eqref{eq:addFil_trick_main_diagram}) don't commute in general in this new setting. The co-connectivity constraints are what we need to make these diagrams commute and hence, to allow us to reduce to the trivial case.

\subsubsection{} Let us now sketch the strategy. Suppose for the moment that we have the following commutative diagram, which is analogous to~\eqref{eq:addFil_trick_main_diagram}, except for the extra conditions
\[
\xymatrix{
	\coLiestar(X)^{\geq 1} \ar[d]_{\addCoFil} \ar[rr]^{\coChev} && \ComAlgstar(\Ran X)^{\geq 2} \\
	\coLiestar(X)^{\geq 1, \coFil^{>0}, \stab} \ar[rr]^{\coChev_{\coFil}} \ar[d]_{\assgr} && \ComAlgstar(\Ran X)^{\geq 2, \coFil^{>0}, \stab} \ar[d]_{\assgr} \ar[u]^{\oblv_{\coFil}} \\
	\coLiestar(X)^{\geq 1, \gr^{>0}, \decay} \ar[rr]^{\coChev_{\gr}} \ar[d]_{\prod} && \ComAlgstar(\Ran X)^{\geq 2, \gr^{>0}, \decay} \ar[d]_{\prod} \\
	\coLiestar(X)^{\geq 1} \ar[rr]^{\coChev} && \ComAlgstar(\Ran X)^{\geq 2}
} \teq\label{eq:addCoFil_trick_main_diagram}
\]
Suppose also that $\oblv_{\coFil}$ preserves factorizability, and that $\assgr$ and $\prod$ are conservative with respect to factorizability.\footnote{Here, by conservativity, we mean that an object satisfies factorizability condition if its image under the functor does.} Then by the same reasoning as in the $\addFil$ trick, to prove that $\coChev\mathfrak{g}$ is factorizable, it suffices to assume that $\mathfrak{g}$ has a trivial $\coLie$-structure. In that case,
\[
\coChev\mathfrak{g} \simeq \Sym^{>0}(\mathfrak{g}[-1]),
\]
and as in the $\Chev$ case, we are done.

In \S\ref{subsec:well-definedness_of_functors}--\S\ref{subsec:factorizability_reflected}, we will carry out the strategy outlined above and conclude the proof of Theorem~\ref{thm:factorizability_coChev}.

\subsection{Well-definedness of functors}
\label{subsec:well-definedness_of_functors}
Before proving that the diagram commutes, we need to first make sense of it. A priori, the functors written in the diagram are not necessarily well-defined. For instance, we have not shown that all the four instances of $\coChev$ land in the correct target categories. Moreover, we also do not know that $\oblv_{\coFil}, \assgr$, and $\prod$ preserve the algebra/co-algebra structures.

The latter question is settled by the following observation, whose proof, which makes use of the stability and decaying conditions to commute limits and tensor products, is straight-forward.
\begin{lem} \label{lem:oblvCoFil_ass-gr_prod_are_monoidal}
For any $n$, the functors
\begin{align*}
	\oblv_{\coFil}&: \Shv(\Ran X)^{\geq n, \coFil^{>0}, \stab} \to \Shv(\Ran X)^{\geq n} \\
	\assgr&: \Shv(\Ran X)^{\geq n, \coFil^{>0}} \to \Shv(\Ran X)^{\geq n, \gr^{>0}} \\
	\prod \simeq \bigoplus&: \Shv(\Ran X)^{\geq n, \gr^{>0}, \decay} \to \Shv(\Ran X)^{\geq n}
\end{align*}
are symmetric monoidal with respect to the $\otimesstar$-monoidal structure on $\Ran X$. In particular, they automatically upgrade to functors between corresponding categories of algebras/co-algebras.
\end{lem}

\subsubsection{} We will now tackle the former question: namely, the various instances of the functor $\coChev$ appeared in~\eqref{eq:addCoFil_trick_main_diagram} land in the correct target categories.

The top and bottom $\coChev$ are the same, and it's easy to see that they land in the correct category using the fact that the shriek-pullback functor is left exact and $\matheur{C}^{\geq n}$ is preserved under limits for any stable infinity category $\matheur{C}$ with a $t$-structure (since $i_{\geq n}$ commutes with limits, see \S\ref{subsubsec:t-structure_convention}).

By the same token, we know that the essential images of $\coChev_{\coFil}$ and $\coChev_{\gr}$ satisfy the co-connectivity assumption (i.e. live in (perverse) cohomological degree $\geq 1$). Thus, it remains to show that they also satisfy the $\stab$ and $\decay$ conditions respectively. For that, first observe that the assertion about $\assgr$ in Lemma~\ref{lem:oblvCoFil_ass-gr_prod_are_monoidal}, combined with the fact that $\assgr$ commutes with limits, gives us a weakened version of the middle square of~\eqref{eq:addCoFil_trick_main_diagram}.
\begin{cor} \label{cor:middle_diagram_cofiltrick_commutes_weakened}
	We have the following commutative diagram
	\[
	\xymatrix{
		\coLiestar(X)^{\geq 1, \coFil^{>0}, \stab} \ar[rr]^<<<<<<<<<<{\coChev_{\coFil}} \ar[d]_{\assgr} && \ComAlgstar(\Ran X)^{\geq 2, \coFil^{>0}} \ar[d]_{\assgr} \\
		\coLiestar(X)^{\geq 1, \gr^{>0}, \decay} \ar[rr]^<<<<<<<<<<{\coChev_{\gr}} && \ComAlgstar(\Ran X)^{\geq 2, \gr^{>0}}
	}
	\]
\end{cor}

Now, by Lemma~\ref{lem:stab_decay_equivalence_Ran_case}, to show that $\coChev_{\coFil}$ and $\coChev_{\gr}$ satisfy the $\stab$ and $\decay$ conditions respectively, it suffices to show that $\coChev_{\gr}$ satisfies the $\decay$ condition. However, this is also a direct consequence of the fact that the shriek-pullback functor is left exact and $\matheur{C}^{>n}$ is preserved under limits (for any stable infinity category $\matheur{C}$ with a $t$-structure). Altogether, we have thus proved that all functors in the diagram~\eqref{eq:addCoFil_trick_main_diagram} above land in the correct categories.

\subsection{Commutative diagrams} We will now proceed to prove that the diagram~\eqref{eq:addCoFil_trick_main_diagram} commutes. First note that we have just settled the commutativity of the middle diagram of~\eqref{eq:addCoFil_trick_main_diagram} at the end of the previous subsection.

\subsubsection{} The commutativity of the bottom diagram of~\eqref{eq:addCoFil_trick_main_diagram} is clear if we know that $\prod$ is symmetric monoidal. However, by Lemma~\ref{lem:direct_sum_and_products_are_the_same_Ran}, we have
\[
	\prod \simeq \bigoplus
\]
and we know that $\bigoplus$ is symmetric monoidal.

\subsubsection{} Finally, to show that the top diagram of~\eqref{eq:addCoFil_trick_main_diagram} commutes, it suffices to show that the following diagram commutes
\[
\xymatrix{
	\coLiestar(X)^{\geq 1} \ar[rr]^{\coChev} && \ComAlgstar(\Ran X)^{\geq 2} \\
	\coLiestar(X)^{\geq 1, \coFil^{>0}, \stab} \ar[u]^{\oblv_{\coFil}} \ar[rr]^{\coChev_{\coFil}} && \ComAlgstar(\Ran X)^{\geq 2, \coFil^{>0}, \stab} \ar[u]^{\oblv_{\coFil}}
} \teq\label{eq:coChev_and_oblv_coFil}
\]
since the composition
\[
	\coLiestar(X)^{\geq 1} \overset{\addCoFil}{\longrightarrow} \coLiestar(X)^{\geq 1, \coFil^{>0}, \stab} \overset{\oblv_{\coFil}}{\longrightarrow} \coLiestar(X)^{\geq 1}
\]
is the identity functor (see also \S\ref{sec:appendix:addFil_oblvFil_commutative}). However, this is clear since the functor
\[
	\oblv_{\coFil}: \Shv(\Ran X)^{\geq n, \coFil^{>0}, \stab} \to \Shv(\Ran X)^{\geq n}
\]
commutes with limits for any $n$, and moreover it is symmetric monoidal with respect to the $\otimesstar$-monoidal structure on $\Shv(\Ran X)$ by Lemma~\ref{lem:oblvCoFil_ass-gr_prod_are_monoidal}.

\subsection{Relation to factorizability} \label{subsec:factorizability_reflected} Using the fact that $\assgr$ is symmetric monoidal and is a conservative functor, it is easy to see that
\[
	\assgr: \ComAlgstar(\Ran X)^{\geq 2, \coFil^{>0}, \stab} \to \ComAlgstar(\Ran X)^{\geq 2, \gr^{>0}, \decay}
\]
reflects factorizability, namely, an object is factorizable if its image is. 

As we already discussed above, we have an equivalence of functors
\[
	\prod \simeq \bigoplus: \ComAlgstar(\Ran X)^{\geq 2, \gr^{>0}, \decay} \to \ComAlgstar(\Ran X)^{\geq 2}.
\]
But now it's clear that $\prod$ reflects factorizability, since $\bigoplus$ does.

Finally, since
\[
	\oblv_{\coFil}: \ComAlgstar(\Ran X)^{\geq 2, \coFil^{>0}, \stab} \to \ComAlgstar(\Ran X)^{\geq 2}
\]
is compatible with $\boxtimes$ (for the same reason that it is compatible with $\otimesstar)$, and moreover $(-)^!$ commutes with limits (being a right adjoint), we see easily that $\oblv_{\coFil}$ preserves factorizability. Thus, we conclude the proof of Theorem~\ref{thm:factorizability_coChev}.

\subsection{Relation to $\coLieshriek(X)$ and $\ComAlgshriek(X)$} In this subsection, we will discuss the various links between objects defined on $X$ such as $\coLieshriek(X)$ and $\ComAlgshriek(X)$ and objects defined on $\Ran X$ such as $\coLiestar(\Ran X)$, $\ComAlgstar(\Ran X)$ and $\Factstar(X)$. This subsection is not used anywhere in the paper. We include it here for the sake of completeness.

\subsubsection{} Recall that on a scheme $X$, there are two symmetric monoidal structures, $\otimes$ and $\otimesshriek$. Thus, we could talk about various algebra/co-algebra objects defined on it
\[
	\Lieast(X), \quad\coLieshriek(X),\quad \ComAlgshriek(X),
\]
where $\Lieast(X)$ (not to be confused with $\Liestar(X)$) is the category of Lie-algebra objects in $\Shv(X)$ with respect to the $\otimes$-monoidal structure, and $\coLieshriek(X)$ (resp. $\ComAlgshriek(X)$) is the category of coLie-algebra (resp. commutative algebra) objects in $\Shv(X)$ with respect to the $\otimesshriek$-monoidal structure.

\subsubsection{} The following observations are straightforward, and are both based on the fact that the functors
\[
	\ins_X^*: \Shv(\Ran X)^{\otimesstar} \to \Shv(X)^{\otimes} \qquad\text{and}\qquad \ins_X^!: \Shv(\Ran X)^{\otimesstar} \to \Shv(X)^{\otimesshriek}
\]
are symmetric monoidal, where
\[
	\ins_X: X \to \Ran X
\]
is the diagonal embedding.

\begin{lem}
We have a pair of adjoint functors
\[
	\ins_{X}^* :\Liestar(\Ran X) \rightleftarrows \Lieast(X): \ins_{X*}
\]
which induces an equivalence of categories
\[
	\Liestar(X) \simeq \Lieast(X),
\]
where the LHS denotes the full-subcategory of $\Liestar(\Ran X) = \Lie(\Shv(\Ran X)^{\otimesstar})$ consisting of $\Lie$-algebras whose underlying sheaves are supported on the diagonal $X$ of $\Ran X$. 
\end{lem}

\begin{lem}
We have a pair of adjoint functors
\[
	\ins_{X!}: \coLieshriek(X) \rightleftarrows \coLiestar(\Ran X): \ins_X^!
\]
which induces an equivalence of categories
\[
	\coLieshriek(X) \simeq \coLiestar(X),
\]
where the RHS denotes the full-subcategory of $\coLiestar(\Ran X) = \coLie(\Shv(\Ran X)^{\otimesstar})$ consisting of $\coLie$-coalgebras whose underlying sheaves are supported on the diagonal $X$ of $\Ran X$.
\end{lem}

\subsubsection{} We also have the following functor
\[
	\ins_X^!: \ComAlgstar(\Ran X) \to \ComAlgshriek(X)
\]
which commutes with limits. Thus, we get a pair of adjoint functors
\[
	\ins_{X?}: \ComAlgshriek(X) \rightleftarrows \ComAlgstar(\Ran X):\ins_X^!. \teq \label{eq:adjunction_ComAlg_X_RanX}
\]

We have the following result from~\cite[Thm. 5.6.4]{gaitsgory_weils_2014}.
\begin{thm} \label{thm:GL_equiv_cats_comalg}
The pair of adjoint functors in~\eqref{eq:adjunction_ComAlg_X_RanX} induces an equivalence of categories
\[
	\ComAlgshriek(X) \simeq \Factstar(X).
\]
\end{thm}

\subsubsection{} The first link between $\coLieshriek(X), \coLiestar(X), \ComAlgshriek(X), \ComAlgstar(\Ran X)$ and $\Factstar(X)$ is given by the following 

\begin{prop}
The following diagram commutes
\[
\xymatrix{
	\coLieshriek(X) \ar[d]_{\coChev} & \coLiestar(X) \ar[l]^\simeq_{\ins_X^!} \ar[d]_{\coChev} \\
	\ComAlgshriek(X) & \ComAlgstar(\Ran X) \ar[l]_{\ins_X^!}
} \teq\label{eq:coChev_commutes_uppershriek}
\]
\end{prop}
\begin{proof}
The result is straightforward due to the fact that $\ins_X^!$ commutes with limits and that it's monoidal.
\end{proof} 

\subsubsection{} The second link, and also the more interesting one, is given by the following
\begin{prop}
We have the following commutative diagram
\[
\xymatrix{
	\coLieshriek(X)^{\geq 1} \ar[d]_{\coChev} \ar[r]_\simeq^{\ins_{X!}} & \coLiestar(X)^{\geq 1} \ar[d]_{\coChev} \\
	\ComAlgshriek(X) \ar[r]^{\ins_{X?}} &  \Factstar(X)
}
\]
\end{prop}
\begin{proof}
By adjunction, for any $\mathfrak{g} \in \coLieshriek(X)$, we have a natural map
\[
	\ins_{X?} \circ \coChev \to \coChev \circ \ins_{X!}
\]
between objects in $\ComAlgstar(\Ran X)$. Now, we know from Theorem~\ref{thm:GL_equiv_cats_comalg} that the LHS is factorizable. Moreover, when $\mathfrak{g}\in \coLieshriek(X)^{\geq 1}$, we know from Theorem~\ref{thm:factorizability_coChev} that the RHS is also factorizable. Thus, to show that the map above is an equivalence when $\mathfrak{g} \in \coLieshriek(X)^{\geq 1}$, it suffices to show that they are equivalent on the diagonal. However, that is clear from~\eqref{eq:coChev_commutes_uppershriek} and we are done.
\end{proof}

\section{Interactions between various functors on the Ran space}
In this section, we investigate how the various functors operating on sheaves on the $\Ran$ spaces interact with each other. The highlights are Theorem~\ref{thm:coChev_and_C^*_c(Ran)}, which says that $\coChev$ is compatible with $C^*_c(\Ran X, -)$ under some co-connectivity assumption, and Theorem~\ref{thm:Chev_coChev_and_D_Ran} which shows how the functor of taking Koszul duality exchanges $\coChev$ and $\Chev$ under some connectivity assumption.

\subsection{$C^*_c(\Ran X, -)$ and $\coChev$}
In this subsection, we will prove Theorem~\ref{thm:intro:coChev_and_C^*_c(Ran)}, which gives us a criterion for the commutativity of the functor $\coChev$ and the functor $C^*_c(\Ran X, -)$. Note that it has been proved in~\cite{francis_chiral_2011} that $\Chev$ always commutes with $C^*_c(\Ran X, -)$. The main reason is that $C^*_c(\Ran X, -)$ is continuous and monoidal with respect to the $\otimesstar$-monoidal structure on $\Shv(\Ran X)$ and the usual monoidal structure on $\Vect$. As before, our main difficulty comes from the fact that $\coChev$ is defined as a limit, and for that to behave well with respect to $C^*_c(\Ran X, -)$, we need to impose a certain co-connectivity assumption.

\subsubsection{} Throughout this subsection, $X$ will be assumed to be a proper scheme of pure dimension $d$.

\begin{thm} \label{thm:coChev_and_C^*_c(Ran)}
For any $\mathfrak{g} \in \coLiestar(X)^{\geq d+1}$, the natural map
\[
	C^*_c(\Ran X, \coChev \mathfrak{g}) \to \coChev(C^*_c(X, \mathfrak{g}))
\]
is an equivalence.\footnote{Since $\Supp \mathfrak{g} \subset X \subset \Ran X$, we have $C^*_c(\Ran X, \mathfrak{g}) \simeq C^*_c(X, \mathfrak{g})$}.
\end{thm}

After some preparation, the actual proof of the theorem will be carried out in~\S\ref{subsubsec:proof_of_coChev_and_C^*_c}. We start with the following elementary lemma whose proof is immediate.

\begin{lem} \label{lem:lim_vs_colim_elementary}
Let $F: \mathbb{N} \times \mathbb{N}^{\op} \to \matheur{C}$ be a functor. Assume that there exists $N \in \mathbb{N}$ such that for all $i, j > N$, the following maps
\[
	F(i, j) \to F(i+1, j) \qquad\text{and}\qquad F(i, j) \to F(i, j-1)
\]
are equivalences, i.e. $F|_{\mathbb{N}_{>N} \times \mathbb{N}_{>N}^{\op}}$ factors through the maximal sub-groupoid of $\matheur{C}$. Then
\[
	\colim_{i\in \mathbb{N}} \lim_{j \in \mathbb{N}^\op} F(i, j) \simeq \lim_{j\in \mathbb{N}^\op} \colim_{i\in \mathbb{N}} F(i, j) \simeq F(N, N),
\]
assuming that all limits and colimits exist.
\end{lem}

\begin{cor} \label{cor:lim_vs_colim}
Let $\matheur{C}$ be a stable $\infty$-category equipped with a right-separated $t$-structure and assume also that filtered colimits are exact with respect to the $t$-structure. Let
\[
	F: \mathbb{N} \times \mathbb{N}^\op \to \matheur{C}
\]
such that for any $c$, the functor $\tr_{<c} \circ F$ satisfies the conditions of Lemma~\ref{lem:lim_vs_colim_elementary}. Then
\[
	\colim_{i\in \mathbb{N}} \lim_{j \in \mathbb{N}^\op} F(i, j) \simeq \lim_{j\in \mathbb{N}^\op} \colim_{i\in \mathbb{N}} F(i, j),
\]
assuming that all limits and colimits make sense.
\end{cor}
\begin{proof}
The separatedness condition implies that it suffices to prove that for each integer $c$, the following map is an equivalence
\[
	\tr_{<c} \colim_{i\in \mathbb{N}} \lim_{j \in \mathbb{N}^\op} F(i, j) \simeq \tr_{<c} \lim_{j\in \mathbb{N}^\op} \colim_{i\in \mathbb{N}} F(i, j).
\]
Commuting $\tr_{<c}$ pass the colimit and limit, the equivalence is a direct consequence of Lemma~\ref{lem:lim_vs_colim_elementary} above. Note that here, we only use the exactness of filter colimits ($\tr_{<0}$ commutes with limits since it's a right adjoint).
\end{proof}

We will apply the discussion above to the situation at hand. 

\subsubsection{Truncated Ran space} For any scheme $X$ and any positive integer $n$, we define
\[
	\Ran^{\leq n} X \simeq \colim_{\substack{I \in \fSet^{\surj} \\ |I| \leq n}} X^I.
\]
Then
\[
	\Ran X \simeq \colim \Ran^{\leq n} X \simeq \colim (X \to \Ran^{\leq 2} X \to \Ran^{\leq 3} X \to \cdots),
\]
and hence, for any $\matheur{F} \in \Shv(\Ran X)$,
\[
	C^*_c(\Ran X, \matheur{F}) \simeq \colim_n C^*_c(\Ran^{\leq n} X, \matheur{F}|_{\Ran^{\leq n}  X}).
\]

The following observation, which gives the link among the cohomology groups
\[
	C^*_c(\Ran^{\leq n}X, \matheur{F}|_{\Ran^{\leq n}X})
\]
for various $n$'s, comes from~\cite[Cor. 9.1.4]{gaitsgory_atiyah-bott_2015}.
\begin{lem} \label{lem:difference_in_XcircI}
We have the following natural equivalence
\[
	C^*(\oversetsupscript{X}{\circ}{I}, \matheur{F}|_{\oversetsupscript{X}{\circ}{I}})_{\Sigma_I} \simeq \coFib(C^*_c(\Ran^{\leq |I|-1} X, \matheur{F}|_{\Ran^{\leq |I|-1} X}) \to C^*_c(\Ran^{\leq |I|} X, \matheur{F}|_{\Ran^{\leq |I|} X})).
\]
\end{lem}

\subsubsection{$\coChev$ as a sequential limit} When
\[
	\mathfrak{g} \in \coLiestar(X)^{\geq d+1},
\]
using the $\addCoFil$ trick~\eqref{eq:addCoFil_trick_main_diagram}, we can also express $\coChev \mathfrak{g}$ as a sequential limit
\[
	\coChev \mathfrak{g} \simeq \oblv_{\coFil}\coChev_{\coFil} \addCoFil \mathfrak{g} \simeq \lim_i (\coChev_{\coFil} \addCoFil \mathfrak{g})_i.
\]
Where $(\coChev_{\coFil} \addCoFil \mathfrak{g})_i$ is the $i$-th step in the co-filtration, and moreover
\[
	\Fib((\coChev_{\coFil} \addCoFil \mathfrak{g})_i \to (\coChev_{\coFil} \addCoFil \mathfrak{g})_{i-1}) \simeq \Sym^i (\mathfrak{g}[-1]),
\]
where $\Sym$ is formed using the $\otimesstar$-monoidal structure on $\Shv(\Ran X)$.

\subsubsection{} For brevity's sake, we will denote
\[
	\coChev^i \mathfrak{g} = (\coChev_{\coFil} \addCoFil \mathfrak{g})_i
\]
and so we have
\[
	\coChev \mathfrak{g} \simeq \lim_i \coChev^i \mathfrak{g}
\]
and
\[
	\Fib(\coChev^i \mathfrak{g} \to \coChev^{i-1}\mathfrak{g}) \simeq \Sym^i (\mathfrak{g}[-1]), \teq\label{eq:fib_coChev_Sym}
\]
where $\Sym$ is formed using the $\otimesstar$-monoidal structure on $\Shv(\Ran X)$.

\subsubsection{} For $\mathfrak{g} \in \coLiestar(X)^{\geq d+1}$, consider the following functor
\begin{align*}
	F: \mathbb{N} \times \mathbb{N}^\op &\to \Vect \teq\label{eq:The_F_C_Ran_coChev} \\
	(i, j) &\mapsto C^*_c(\Ran^{\leq i} X, \coChev^j \mathfrak{g}) \simeq \coChev^j C^*_c(X, \mathfrak{g})
\end{align*}
where the equivalence on the second line is due to the fact that $\coChev^j$ is computed as a finite limit for each $j$.

The goal now is to show that $F$ satisfies the conditions stated in Corollary~\ref{cor:lim_vs_colim}. We start with a couple of cohomological estimates.

\begin{lem} \label{lem:estimates_coChev_and_Sym_on_Ran}
For any $\mathfrak{g} \in \coLiestar(X)^{\geq d+1}$ and any non-negative integer $i$,
\[
	\Supp \coChev^i \mathfrak{g} \subset \Ran^{\leq i} X
\]
and for all non-empty finite set $I$ such that $|I| \leq i$, $(\Sym^i (\mathfrak{g}[-1]))|_{\oversetsupscript{X}{\circ}{I}}$ lives in perverse cohomological degrees $\geq i(d+2)$.
\end{lem}
\begin{proof}
This follows directly from~\eqref{eq:otimesstar_explicit_formula} and the fact that $!$-pullbacks are left exact with respect to the perverse $t$-structure.
\end{proof}

\begin{cor} \label{cor:estimates_C^*_XI_coChevj}
For any $\mathfrak{g} \in \coLiestar(X)^{\geq d+1}$, any non-empty finite set $I$, and any positive integer $j$, $(\coChev^j \mathfrak{g})|_{\oversetsupscript{X}{\circ}{I}}$ lives in perverse cohomological degrees $\geq |I|(d+2)$. In particular,
\[
	C^*(\oversetsupscript{X}{\circ}{I}, (\coChev^j\mathfrak{g})|_{\oversetsupscript{X}{\circ}{I}})_{\Sigma_I}
\]
lives in cohomological degrees $\geq 2|I|$.
\end{cor}
\begin{proof}
Since $C^*(\oversetsupscript{X}{\circ}{I}, -)[-d|I|]$ is $t$-left exact, the second statement follows from the first. Now, when $j < |I|$, then there is nothing to prove since everything vanishes. For $j \geq |I|$, we have the following sequence of sheaves
\[
	\coChev^j \mathfrak{g}|_{\oversetsupscript{X}{\circ}{I}} \to \coChev^{j-1} \mathfrak{g}|_{\oversetsupscript{X}{\circ}{I}} \to \cdots \to \coChev^{|I|} \mathfrak{g}|_{\oversetsupscript{X}{\circ}{I}} \to \coChev^{|I| - 1} \mathfrak{g}|_{\oversetsupscript{X}{\circ}{I}} \simeq 0.
\]
Inducting on $k \in \{|I|, \dots, j\}$, using the fact that the $k$-th fiber of this sequence is $\Sym^k (\mathfrak{g}[-1])|_{\oversetsupscript{X}{\circ}{I}}$ (see~\eqref{eq:fib_coChev_Sym}) and the estimates in Lemma~\ref{lem:estimates_coChev_and_Sym_on_Ran} concludes the proof.
\end{proof}

\begin{lem} \label{lem:estimates_C^*_Ran_leq_i_coChevj}
For any $\mathfrak{g} \in \coLiestar(X)^{\geq d+1}$ and any pair of positive integers $i, j$,
\[
	C^*_c(\Ran^{\leq i} X, \Sym^j (\mathfrak{g}[-1])|_{\Ran^{\leq i} X})
\]
lives in cohomological degrees $\geq 2j$.
\end{lem}
\begin{proof}
Consider the following sequence of chain complexes
\[
	C^*_c(X, \Sym^j (\mathfrak{g}[-1])|_{X}) \to C^*_c(\Ran^{\leq 2} X, \Sym^j (\mathfrak{g}[-1])|_{\Ran^{\leq 2} X}) \to \dots \to C^*_c(\Ran^{\leq i} X, \Sym^j(\mathfrak{g}[-1])|_{\Ran^{\leq i} X}),
\]
with the $k$-th co-fiber being
\[
	C^*(\oversetsupscript{X}{\circ}{k}, \Sym^j(\mathfrak{g}[-1])|_{\oversetsupscript{X}{\circ}{k}})_{\Sigma_k}, \quad k\in \{1, \dots, i\}
\]
by Lemma~\ref{lem:difference_in_XcircI}.\footnote{Since $X$ is assumed to be proper throughout this subsection, our statement is valid also for the case $k=1$.} By Lemma~\ref{lem:estimates_coChev_and_Sym_on_Ran}, we see that this chain complex lives in cohomological degrees 
$\geq j(d+2)-kd$ when $k\leq j$ and vanishes otherwise. Thus, in particular, it lives in cohomological degrees $\geq 2j$. Inducting on $k \in \{1, \dots, i\}$, we conclude the proof.
\end{proof}

\begin{prop} \label{prop:commuting_limit_colim_C_Ran_Chev}
When $\mathfrak{g} \in \coLiestar(X)^{\geq d+1}$, the functor $F$ considered at~\eqref{eq:The_F_C_Ran_coChev} satisfies the conditions stated in Corollary~\ref{cor:lim_vs_colim}. In particular, we have a natural equivalence
\[
	\colim_i \lim_j C^*_c(\Ran^{\leq i} X, \coChev^j \mathfrak{g}|_{\Ran^{\leq i} X}) \simeq \lim_j \colim_i C^*_c(\Ran^{\leq i} X, \coChev^j \mathfrak{g}|_{\Ran^{\leq i} X}).
\]
\end{prop}
\begin{proof}
This is a direct consequence of Corollary~\ref{cor:estimates_C^*_XI_coChevj} and Lemma~\ref{lem:estimates_C^*_Ran_leq_i_coChevj}.
\end{proof}

\subsubsection{}\label{subsubsec:proof_of_coChev_and_C^*_c} With these observations, we are ready for the proof of Theorem~\ref{thm:coChev_and_C^*_c(Ran)}.
\begin{proof}[Proof of Theorem~\ref{thm:coChev_and_C^*_c(Ran)}] 	
We have
\begin{align*}
C^*_c(\Ran X, \coChev \mathfrak{g}) 
&\simeq \colim_i C^*_c(\Ran^{\leq i} X, \lim_i \coChev^j \mathfrak{g}|_{\Ran^{\leq i} X}) \\
&\simeq \colim_i \lim_j C^*_c(\Ran^{\leq i} X, \coChev^j \mathfrak{g}|_{\Ran^{\leq i}X}) 
\teq\label{eq:lim_j_C(Ran<=i)} \\
&\simeq \lim_j \colim_i C^*_c(\Ran^{\leq i} X, \coChev^j \mathfrak{g}|_{\Ran^{\leq i}X})
\teq\label{eq:commuting_limit_colimit}\\
&\simeq \lim_j C^*_c(\Ran X, \coChev^j \mathfrak{g}) \\
&\simeq \lim_j \coChev^j C^*_c(X, \mathfrak{g}) 
\teq\label{eq:commute_coChevj_C(Ran)}\\
&\simeq \coChev C^*_c(X, \mathfrak{g}).
\teq\label{eq:add_coFil_pt}
\end{align*}
Here,~\eqref{eq:lim_j_C(Ran<=i)} is due to the fact that $C^*_c(\Ran^{\leq i} X, -)$ is a finite colimit of functors of the form $C^*_c(X^I, -)$, each of which commutes with limits since $X$ is proper. Moreover,~\eqref{eq:commuting_limit_colimit} is due to~Proposition~\ref{prop:commuting_limit_colim_C_Ran_Chev} and~\eqref{eq:commute_coChevj_C(Ran)} is due to the fact that $\coChev^j$ is a finite limit and $\mathfrak{g}$ is supported only on $X$. Finally,~\eqref{eq:add_coFil_pt} is obtained by applying the $\addCoFil$ trick to the case of $\Vect$.
\end{proof}

\begin{rmk}
In the last step~\eqref{eq:add_coFil_pt}, we need $\mathfrak{g}$ to live in perverse cohomological degrees $\geq d + 1$ so that $C^*_c(X, \mathfrak{g}) \simeq C^*(X, \mathfrak{g})$ lives in cohomological degrees $\geq 1$, which is needed to apply the $\addCoFil$ trick. Here, $X=\pt$ in ~\eqref{eq:addCoFil_trick_main_diagram}.
\end{rmk}

\subsection{Verdier duality}
Before studying the link between $\Chev$ and $\coChev$, we start with a quick recollection of Verdier duality on prestacks along with various useful properties. The main reference is~\cite{gaitsgory_atiyah-bott_2015}. We only use the very basic properties of $D_{\Ran}$.

\subsubsection{} Let $\matheur{Y}$ be a prestack such that the diagonal map
\[
	\diag_{\matheur{Y}}: \matheur{Y} \to \matheur{Y} \times \matheur{Y}
\]
is pseudo-proper. For $\matheur{F}, \matheur{G} \in \Shv(\matheur{Y})$, by a pairing between them, we shall mean a map
\[
	\matheur{F}\boxtimes \matheur{G} \to \diag_{\matheur{Y}!} \omega_{\matheur{Y}}.
\]
We define the Verdier dual $D_\matheur{Y} \matheur{G}$ of $\matheur{G}$ to be the object representing the functor
\[
	\matheur{F} \mapsto \Hom(\matheur{F} \boxtimes \matheur{G}, \diag_{\matheur{Y}!} \omega_{\matheur{Y}}).
\]
Namely, we have the following natural equivalence
\[
	\Hom(\matheur{F}, D_{\matheur{Y}} \matheur{G}) \simeq \Hom(\matheur{F} \boxtimes \matheur{G}, \diag_{\matheur{Y}!} \omega_{\matheur{Y}}).
\]

The following lemma is immediate from the definition.
\begin{lem} \label{lem:D_turns_colimits_to_limits}
Let $\matheur{F} \in \Shv(\matheur{Y})$, such that
\[
	\matheur{F} \simeq \colim_{i \in \matheur{I}} \matheur{F}_i.
\]
Then
\[
	D_{\matheur{Y}} \matheur{F} \simeq \lim_{i \in \matheur{I}^{\op}} D_{\matheur{Y}} \matheur{F}_i.
\]
\end{lem}

\subsubsection{} We will now study the link between Verdier duality and $\boxtimes$.

\begin{prop}
\label{prop:boxtimes_finitary_pseudo-scheme}
Let $\matheur{Y}_1$ and $\matheur{Y}_2$ be finitary pseudo-schemes, and $\matheur{F}_i \in \Shv(\matheur{Y}_i)$ for $i\in \{1, 2\}$. Then, we have a natural equivalence
\[
	D_{\matheur{Y}_1} \matheur{F}_1 \boxtimes D_{\matheur{Y}_2} \matheur{F}_2 \simeq D_{\matheur{Y}_1 \times \matheur{Y}_2} (\matheur{F}_1 \boxtimes \matheur{F}_2).
\]
\end{prop}
\begin{proof}
First, note that the result holds when both $\matheur{Y}_1$ and $\matheur{Y}_2$ are schemes.
	
For the general case of finitary pseudo-schemes, we write
\[
	\matheur{Y}_1 \simeq \colim_i Y_{1i}\qquad\text{and}\qquad \matheur{Y}_2 \simeq \colim_j Y_{2j}.
\]
Then,
\[
	\matheur{F}_1 \simeq \colim_{i} \ins_{1i!} \ins_{1i}^! \matheur{F}_1 \qquad\text{and}\qquad \matheur{F}_2 \simeq \colim_j \ins_{2j!} \ins_{2j}^! \matheur{F}_2.
\]
Thus,
\begin{align*}
	D_{\matheur{Y}_1 \times \matheur{Y}_2}(\matheur{F}_1 \boxtimes \matheur{F}_2) 
	&\simeq D_{\matheur{Y}_1 \times \matheur{Y}_2} \colim_{i, j} (\ins_{1i} \times \ins_{2j})_! (\ins_{1i}\times \ins_{2j})^! (\matheur{F}_1 \boxtimes \matheur{F}_2) \\
	&\simeq \lim_{i, j} (\ins_{1i} \times \ins_{2j})_! D_{Y_{1i} \times Y_{2j}} (\ins_{1i}^! \matheur{F}_1 \boxtimes \ins_{2j}^! \matheur{F}_2) \teq \label{eq:D_and_tensor_finitary_!pushforward_1} \\
	&\simeq \lim_{i, j} (\ins_{1i} \times \ins_{2j})_! (D_{Y_{1i}} \ins_{1i}^! \matheur{F}_1  \boxtimes D_{Y_{2j}} \ins_{2j}^! \matheur{F}_2) \label{eq:D_and_tensor_scheme_case} \teq \\
	&\simeq (\lim_i \ins_{1i!} D_{Y_{1i}} \ins_{1i}^! \matheur{F}_1) \boxtimes (\lim_j \ins_{2j!} D_{Y_{2j}} \ins_{1j}^! \matheur{F}_2) \label{eq:D_and_tensor_finite_limit} \teq \\
	&\simeq (D_{\matheur{Y}_1} \colim_i \ins_{1i!} \ins_{1i}^! \matheur{F}_1)\boxtimes (D_{\matheur{Y}_2} \ins_{2j!} \ins_{2j}^! \matheur{F}_2) \teq \label{eq:D_and_tensor_finitary_!pushforward_2} \\
	&\simeq D_{\matheur{Y}_1} \matheur{F}_1 \boxtimes D_{\matheur{Y}_2}\matheur{F}_2.
\end{align*}
Here,~\eqref{eq:D_and_tensor_scheme_case} is due to the fact that the statement we are trying to prove holds for the case of schemes,~\eqref{eq:D_and_tensor_finite_limit} is due to the fact that the limits we are taking are all finite (due to the finitary assumption), and finally, both~\eqref{eq:D_and_tensor_finitary_!pushforward_1} and \eqref{eq:D_and_tensor_finitary_!pushforward_2} are due to Lemma~\ref{lem:D_turns_colimits_to_limits} and Proposition~\ref{prop:D_finitary_!pushforward} below.
\end{proof}

\begin{prop}\label{prop:D_finitary_!pushforward}
Let $f: \matheur{Y}_1 \to \matheur{Y}_2$ be a finitary pseudo-proper map between pseudo-schemes, each having a finitary diagonal. Then, the natural transformation
\[
	f_!\circ D_{\matheur{Y}_1} \to D_{\matheur{Y}_2}\circ f_!
\]
is an equivalence.
\end{prop}
\begin{proof}
	See~\cite[Cor. 7.5.6]{gaitsgory_atiyah-bott_2015}.
\end{proof}

\begin{rmk}
One direct corollary of this proposition is the fact that for any sheaf $\matheur{F} \in \Shv(X)$, we have the following natural equivalence
\[
	\delta_{X!} D_X \matheur{F} \simeq D_{\Ran X} \delta_{X!} \matheur{F}.
\]
\end{rmk}

\begin{cor} \label{cor:D_Ran_commutes_otimesstar_finite_support}
Let $\matheur{F}_1, \matheur{F}_2, \cdots, \matheur{F}_k \in \Shv(\Ran X)$ with finite supports, i.e. there exists an $n$ such that all the $\matheur{F}_i$'s are $!$-pushforward of sheaves on $\Ran^{\leq n} X$. Then, we have the following natural equivalence
\[
	D_{\Ran X} (\matheur{F}_1 \otimesstar\matheur{F}_2 \otimes \cdots \otimesstar \matheur{F}_k) \simeq (D_{\Ran X} \matheur{F}_1) \otimesstar (D_{\Ran X} \matheur{F}_2) \otimesstar \cdots \otimesstar (D_{\Ran X} \matheur{F}_k).
\]
\end{cor}
\begin{proof}
Since the sheaves involved have finite supports, their box-tensor commutes with Verdier duality on $\Ran^{\leq n} X$, by Proposition~\ref{prop:boxtimes_finitary_pseudo-scheme}. Since $\Ran^{\leq n} X \to \Ran X$ is finitary pseudo-proper, Proposition~\ref{prop:D_finitary_!pushforward} implies that their box-tensor also commutes with Verdier duality on $\Ran X$. Finally, using the fact that the union map is finitary pseudo-proper, Proposition~\ref{prop:boxtimes_finitary_pseudo-scheme} then implies that $\otimesstar$ of these sheaves also commutes with Verdier duality on the $\Ran$ space.
\end{proof}

\subsection{$\Chev$, $\coChev$, and $D_{\Ran X}$} 
We will now turn to Theorem~\ref{thm:intro:Chev_coChev_and_D_Ran}, which provides a link between the two functors $\Chev$ and $\coChev$ via the functor of taking Verdier duality on the $\Ran$ space.
\begin{thm} \label{thm:Chev_coChev_and_D_Ran}
Let $\mathfrak{g}\in \Liestar(X)^{\leq -1}$. Then we have a natural equivalence
\[
\coChev(D_{X} \mathfrak{g}) \simeq D_{\Ran X} \Chev(\mathfrak{g}),
\]
of objects in $\ComAlgstar(\Ran X)$, where $D_{\Ran X}$ is the functor of taking Verdier duality on $\Ran X$.
\end{thm}

Note that this is the only place we use Verdier duality on the Ran space. However, we essentially use it in a rather minimal way: not much besides the definition itself.

\begin{proof}
We will employ ideas originated from the $\addFil$ and $\addCoFil$ tricks (see also \S\ref{sec:appendix:addFil_trick}). First, observe that for any $\mathfrak{g} \in \Liestar(X)$, we have a canonical equivalence
\[
	\addCoFil D_{\Ran X} \mathfrak{g} \simeq D_{\Ran X} \addFil \mathfrak{g}.
\]
We use $\Chev^i \mathfrak{g}$ and $\coChev^i D_{\Ran X}\mathfrak{g}$ to denote the $i$-th piece in the filtration/co-filtration of $\Chev (\addFil\mathfrak{g})$ and $\coChev (\addCoFil D_{\Ran X} \mathfrak{g})$ respectively.

From \S\ref{sec:appendix:addFil_trick} and the top part of the commutative diagram~\eqref{eq:addCoFil_trick_main_diagram}, we have the following natural equivalences
\begin{align*}
	\Chev\mathfrak{g} &\simeq \colim_i \Chev^i \mathfrak{g}, \\
	\coChev (D_{\Ran X}\mathfrak{g}) &\simeq \lim_i \coChev^i (D_{\Ran X} \mathfrak{g}).
\end{align*}
At the same time, by Lemma~\ref{lem:D_turns_colimits_to_limits}, we know that
\[
	D_{\Ran X} \colim_i \Chev^i \mathfrak{g} \simeq \lim_i D_{\Ran X} \Chev^i \mathfrak{g}.
\]
Thus, it suffices to show that
\[
	D_{\Ran X} \Chev^i \mathfrak{g} \simeq \coChev^i D_{\Ran X}\mathfrak{g}.
\]
Now, it's an immediate consequence of Corollary~\ref{cor:D_Ran_commutes_otimesstar_finite_support}.
\end{proof}

\begin{cor}
Let $\mathfrak{g}\in \Liestar(X)^{\leq -1}$. Then $D_{\Ran X} \Chev(\mathfrak{g})$ is a factorizable commutative algebra on $\Ran X$.
\end{cor}
\begin{proof}
This is a direct consequence of Theorem~\ref{thm:Chev_coChev_and_D_Ran} and Theorem~\ref{thm:factorizability_coChev}. 
\end{proof}

\subsection{$\coChev$ and open embeddings} We end the section with the following simple observation.

\begin{prop} \label{prop:coChev_open_embedding} Let
\[
	j: X' \to X
\]
be an open embedding of schemes, which induces an open embedding of prestacks
\[
	j_{\Ran}: \Ran X' \to \Ran X.
\]
Then for any $\mathfrak{g}' \in \coLiestar(X')$, we have the following natural equivalence
\[
	(j_{\Ran})_* \coChev(\mathfrak{g}') \simeq \coChev(j_* \mathfrak{g}').
\]
\end{prop}
\begin{proof}[Proof (Sketch)]
The result is a direct consequence of the fact that $(j_{\Ran})_*$ is symmetric monoidal and commutes with limits. The latter is due to the fact that it is a right adjoint. The former is due to the fact that for any open embeddings of prestacks $f_i: X'_i \to X_i$ and any $\matheur{F}_i \in \Shv(X'_i)$ for $i=1, 2$, we have a natural equivalence
\[
	(f_1 \times f_2)_*(\matheur{F}_1 \boxtimes \matheur{F}_2) \simeq f_{1*}\matheur{F}_1 \boxtimes f_{2*} \matheur{F}_2.
\]
This is in turn a consequence of~\eqref{eq:base_change_lower*_upper!} and the corresponding fact for schemes.
\end{proof}

\section{An application to the Atiyah-Bott formula} \label{sec:application_Atiyah-Bott}
We will now give an application of the results proved so far to the Atiyah-Bott formula. As mentioned in the introduction, these results allow us to simplify the second of the two main steps in the original proofs given in~\cite{gaitsgory_weils_2014} and~\cite{gaitsgory_atiyah-bott_2015}. In what follows, \S\ref{subsec:statement}--\S\ref{subsec:pairing} are intended to orient the readers with the existing results proved in~\cite{gaitsgory_weils_2014} and~\cite{gaitsgory_atiyah-bott_2015},\footnote{Namely, all the results stated in these subsections could be found in~\cite{gaitsgory_weils_2014} or~\cite{gaitsgory_atiyah-bott_2015}. The readers should be warned that we provide a mere overview of the development given in these two papers, with many technical points elided.} whereas the purpose of the last part, \S\ref{subsec:conclusion}, is to explain how the results we've proved so far fit in with the rest. 

\subsection{The statement}\label{subsec:statement}
From now on, $X$ is a smooth and complete curve over an algebraically closed field $k$, and $G$ a smooth, fiber-wise connected group-scheme over $X$, whose generic fiber is semi-simple simply connected. Due to~\cite[Lem. 7.1.1 and Prop A.3.11]{gaitsgory_weils_2014}, we can (and from now on we will) assume that $G$ is semi-simple simply connected over an open dense subset
\[
	j: X' \hookrightarrow X,
\]
and moreover, the fibers of $G$ over any point in $X-X'$ are homologically trivial.

We will also use
\[
	j_{\Ran}: \Ran X' \to \Ran X
\]
to denote the corresponding open embedding on the $\Ran$ space and
\[
	\Gamma_{j_{\Ran}}: \Ran X' \to \Ran X'\times \Ran X
\]
to denote its graph.

\subsubsection{} Let $G_0$ be the split form of $G$. Then it is well-known that
\[
	C^*(BG_0) \simeq \Sym M_0 \teq\label{eq:C^*(BG_0_is_free)}
\]
is a free commutative algebra, for some $M_0\in \Vect$. In the case of $\ell$-adic sheaves in positive characteristic setting, this equivalence is compatible with the geometric Frobenius action, where
\[
	M_0 \simeq \bigoplus_e \Lambda[-2e](-e),
\]
and $e$'s are the exponents of $G_0$.

The assignment $G_0 \mapsto M_0$ is functorial with respect to automorphisms of $G_0$, and hence, for a general $G$ (subject to the assumptions mentioned above), we get a local system
\[
	M \in \Shv(X'),
\]
whose $!$-fiber at each geometric point $x \in X$ is equivalent to $M_0$.

Below is the statement of the Atiyah-Bott formula.
\begin{thm} \label{thm:Atiyah-Bott} Let $G, X$ as above. Then
\begin{enumerate}[(a)]
	\item We have an equivalence
	\[
		C^*(\Bun_G) \simeq \Sym(C^*(X', M)).
	\]
	\item When $k=\Fqbar$, and $X$ and $G$ are defined over $\Fq$, the above equivalence can be chosen to be compatible with the Frobenius actions.
\end{enumerate}
\end{thm}

\subsection{$BG$ and the sheaf $\matheur{B}$}
\label{subsec:BG_and_sheaf_B}
\subsubsection{} The sheaf $\matheur{B}$ that we will now describe encodes the reduced cohomology of $BG$, the relative (over $X$) classifying stack of $G$. For each $I \in \Ran_X(S)$, let $D_I \subset S\times X$ be the corresponding Cartier divisor. Let $BG_I$ denote the Artin stack classifying $G$-bundles over $D_I$ and $f_I: BG_I \to S$ the forgetful map. Then, we define
\[
	\wtilde{\matheur{B}}_{S, I} = D_S(\Fib(f_{I!}f_I^! \Lambda_S \to \Lambda_S)),
\]
where $D_S$ is the functor of taking Verdier duality on $S$. These sheaves, assembled together, give rise to a sheaf (see also~\cite[Prop. 5.4.3]{gaitsgory_weils_2014})
\[
	\wtilde{\matheur{B}} \in \Shv(\Ran X).
\]

\subsubsection{} Note that for any finite set of points $\{x_1, \dots, x_n\} \in (\Ran X)(k)$, the $!$-fiber of $\wtilde{\matheur{B}}$ at this point is
\[
	\coFib\left(\Lambda \to \bigotimes_{i=1}^n C^*(BG_{x_i})\right). \teq\label{eq:B'_!_fiber}
\]

\subsubsection{} Using a variant of the diagonal map
\[
	BG \to BG \times BG,
\]
we can equip $\wtilde{\matheur{B}}$ with the structure of an object in
\[
	\ComAlgstar(\Ran X).
\]
However, we see easily from~\eqref{eq:B'_!_fiber} that $\wtilde{\matheur{B}}$ is not factorizable. The functor $\TakeOut$ developed in~\cite{gaitsgory_atiyah-bott_2015} allows us to remove all the extra components in it and construct out of it a new object $\matheur{B} \in \Factstar(X)$ with the correct $!$-fibers at a point $\{x_1, \dots, x_n\} \in (\Ran X)(k)$
\[
	\bigotimes_{i=1}^n C^*_{\red}(BG_{x_i}).
\]
Moreover, $\matheur{B}$ has the same cohomology along $\Ran X$ as the original sheaf $\wtilde{\matheur{B}}$ (see also~\cite[Cor. 5.3.5]{gaitsgory_atiyah-bott_2015})
\[
	C^*_c(\Ran X, \matheur{B}) \simeq C^*_c(\Ran X, \wtilde{\matheur{B}}).
\]

\subsubsection{$\matheur{B}$ and $\Bun_G$} For every $S \in \Sch$ and $I\in (\Ran X)(S)$, we have a map of prestacks over $S$ by restricting the bundle to the divisor $D_I$
\[
	S \times \Bun_G \to BG_I. \teq\label{eq:ev_Bun_G_BGI}
\]
This induces a map
\[
	\wtilde{\matheur{B}}_{S, I} \to \omega_S \otimes C^*_\red(\Bun_G)
\]
and hence, also a map
\[
	\wtilde{\matheur{B}} \to \omega_{\Ran X} \otimes C^*_{\red}(\Bun_G).
\]
Applying the functor $C^*_c(\Ran X, -)$ and using the fact that $\Ran X$ is homologically contractible, we get a map
\[
	C^*_c(\Ran X, \matheur{B}) \simeq C^*_c(\Ran X, \wtilde{\matheur{B}}) \to C^*_{\red}(\Bun_G). \teq \label{eq:atiyah-bott-canonical-map}
\]

\subsubsection{} Using~\eqref{eq:C^*(BG_0_is_free)} and the assumption we have on $G$, i.e. it has homologically contractible fibers outside of $X'$, one gets an equivalence 
\[
	\matheur{B} \simeq (j_{\Ran X})_* \matheur{B}' \simeq \Sym^{>0}(j_* M) \teq \label{eq:B_is_pushforward_of_B'}
\]
where $\matheur{B}'$ is the restriction of $\matheur{B}$ to $\Ran X'$ and, the symmetric algebra is taken inside $\Shv(\Ran X)$ using the $\otimesstar$-monoidal structure.

\subsubsection{} Using the equivalence~\eqref{eq:B_is_pushforward_of_B'} and the fact that $C^*_c(\Ran X, -)$ commutes with $\Sym$,\footnote{Note that this is a special case of the fact that $C^*_c(\Ran X, -)$ commutes with $\Chev$. And in fact, both are due to the same reasons: that $C^*_c(\Ran X, -)$ is continuous and that it's symmetric monoidal.} we get an explicit presentation of the LHS of~\eqref{eq:atiyah-bott-canonical-map}
\[
	C^*_c(\Ran X, \matheur{B}) \simeq \Sym^{>0} C^*_c(X, j_* M) \simeq \Sym^{>0} C^*(X', M), \teq\label{eq:explicit_C^*_c(Ran X, B)}
\]
which appears in the statement of the Atiyah-Bott formula as stated in Theorem~\ref{thm:Atiyah-Bott}.

\subsubsection{} Now, we are done if we could show that the map in~\eqref{eq:atiyah-bott-canonical-map} is an equivalence. 

\subsection{Affine Grassmannian and the sheaf $\matheur{A}$} \label{subsec:affine_grassmannian_and_sheaf_A}

Unfortunately, one does not know how to directly prove that~\eqref{eq:atiyah-bott-canonical-map} is an equivalence. Instead, \cite{gaitsgory_weils_2014} proceeds with an equivalence of a dual nature, which we will now briefly recall.

\subsubsection{} The main player in this step is the affine Grassmannian, or more precisely, a factorizable version thereof. Let $G$ and $X$ be as above. The factorizable affine Grassmannian of $G$, denoted by $\Gr_{\Ran X'}$, is the prestack whose $S$-points are given by
\[
\Gr_{\Ran X'}(S) = \{(\matheur{P}, I, \alpha)\},
\]
where
\begin{enumerate}[(i)]
	\item $\matheur{P}$ is a $G$-bundle over $S\times X$,
	\item $I$ is a non-empty finite subset of $X'(S)$,
	\item $\alpha$ is a trivialization of $\matheur{P}$ on the complement of the graph of $I$.
\end{enumerate}

\subsubsection{} From the definition, we have the following natural morphism
\[
g: \Gr_{\Ran X'} \to \Ran X',
\]
where we remember only the set $I$, and similarly another natural morphism
\[
u: \Gr_{\Ran X'} \to \Bun_G,
\]
where we remember only the bundle $\matheur{P}$.

\subsubsection{} The map $g$ allows us to define
\[
\wtilde{\matheur{A}}' \simeq \Fib(g_!(\omega_{\Gr_{\Ran X'}})\to \omega_{\Ran X'}) \in \Shv(\Ran X'),
\]
and the map $u$ induces a map at the homology level, namely
\[
C_*^\red(\Gr_{\Ran X'}) \to C_*^\red(\Bun_G). \teq \label{eq:non-abelian-poincare-uniformizing}
\]

Together, we get the following map
\[
C^*_c(\Ran X', \wtilde{\matheur{A}}') \to C_*^\red(\Bun_G). \teq \label{eq:non-abelian-poincare-sheaf-A'}
\]

\subsubsection{} Note that since
\[
\Gr_{\Ran X'} \to \Ran X'
\]
is pseudo-proper, $\wtilde{\matheur{A}}'$ is easy to describe. Namely for any finite set of points $\{x_1, x_2, \dots, x_n\} \subset X(k)$, the $!$-fiber of $\wtilde{\matheur{A}}'$ at this point is
\[
\Fib\left(\bigotimes_{i=1}^n C_*(\Gr_{G_{x_i}}) \to \Lambda\right). \teq\label{eq:!-fiber_A'}
\]

\subsubsection{} Using a variant of the diagonal map
\[
\Gr \to \Gr \times \Gr,
\]
one can equip $\wtilde{\matheur{A}}'$ with the structure of an object in
\[
\ComCoAlgstar(\Ran X').
\]

However, note that the sheaf $\wtilde{\matheur{A}}'$ is not factorizable, since its $!$-fiber, as described in~\eqref{eq:!-fiber_A'}, is too big, i.e. it's not equivalent to
\[
\bigotimes_{i=1}^n C_*^\red(\Gr_{G_{x_i}}). \teq \label{eq:correct-!-fiber-A}
\]
Using a similar reasoning as in the case of $\wtilde{\matheur{B}}$ and $\matheur{B}$, we can construct an object $\matheur{A}' \in \coFactstar(X')$ with the correct $!$-fiber as given in~\eqref{eq:correct-!-fiber-A}, and moreover, $\matheur{A}'$ has the property that
\[
C^*_c(\Ran X', \wtilde{\matheur{A}}') \simeq C^*_c(\Ran X', \matheur{A}'). \teq\label{eq:take_out_unit_A}
\]

\subsubsection{$\matheur{A}$ and $\Bun_G$} The equivalence of a dual nature that we alluded to earlier is given by the following important result (see~\cite[Thm. 3.2.13]{gaitsgory_weils_2014}).

\begin{thm}
	The map~\eqref{eq:non-abelian-poincare-uniformizing}, and hence~\eqref{eq:non-abelian-poincare-sheaf-A'}, is an equivalence.
\end{thm}

This theorem is essentially a result about the homological contractibility of the space of rational map (maps that are defined only on an open subset) from $X$ to $G$. An earlier version of this was proved in~\cite{gaitsgory_contractibility_2012}. Together with~\eqref{eq:take_out_unit_A} we have the following
\begin{prop} \label{prop:non-abelian-poincare_sheaf_A}
We have a natural equivalence
\[
	C^*_c(\Ran X', \matheur{A}') \simeq C^\red_*(\Bun_G).
\]
\end{prop}
\subsection{Pairing}
\label{subsec:pairing}
We will now describe how the equivalence given by Proposition~\ref{prop:non-abelian-poincare_sheaf_A} helps us show that \eqref{eq:atiyah-bott-canonical-map} is an equivalence.

\subsubsection{} For any schemes $S, S'\in \Sch$ and any non-empty finite subsets $I \subset X(S)$ and $I'\subset X'(S')$, we have a natural map (which is just a more elaborate variant of~\eqref{eq:ev_Bun_G_BGI})
\[
	\Gr_{I'} \times S \to \Bun_G \times S' \times S \to S'\times BG_{I},
\]
which induces a map
\[
	\matheur{A}' \boxtimes \matheur{B} \to \omega_{\Ran X' \times \Ran X},
\]
and hence, a pairing (using $\TakeOut$)
\[
	\matheur{A}' \boxtimes \matheur{B} \to \Gamma_{j_{\Ran}!}\omega_{\Ran X'}.
\]

\subsubsection{} Restricting this map to $\Ran X' \times \Ran X'$ gives us the following map
\[
	\matheur{A}' \boxtimes \matheur{B}' \to (\delta_{\Ran X'})_! \omega_{\Ran X'},
\]
and hence, using the definition of Verdier duality, a map
\[
	\matheur{B}' \to D_{\Ran X'}\matheur{A}' \teq\label{eq:B'_to_DA'}
\]
between objects in $\ComAlgstar(\Ran X')$.

\subsubsection{} It is proved, in fact twice (using different methods), in \S17 and \S18 of \cite{gaitsgory_atiyah-bott_2015}, that the restriction of~\eqref{eq:B'_to_DA'} to the diagonal $X'$ of $\Ran X'$ is an equivalence. Namely, we have
\[
	\matheur{B}'|_{X'} \simeq (D_{\Ran X'} \matheur{A}')|_{X'} \teq\label{eq:B'_eq_DA'_on_X'}.
\]

\subsection{The last steps}
\label{subsec:conclusion}
The results that we have just proved in this paper appear in two places in the concluding steps, which are given by Proposition~\ref{prop:DA'_factorizable} and~\ref{prop:local_dual_implies_global_dual}. Together, they imply the Atiyah-Bott formula.

\begin{prop} \label{prop:DA'_factorizable}
$D_{\Ran X'} \matheur{A}'$ is factorizable, i.e.
\[
	D_{\Ran X'} \matheur{A}' \in \Factstar(X') \subset \ComAlgstar(\Ran X').
\]
\end{prop}
\begin{proof}
It is well-known that for a split semi-simple simply connected group $G_0$, $C_*^\red(\Gr_{G_0}, \Lambda)$ lives in cohomological degrees $\leq -2$. Using the fact that
\[
	\Gr_{\Ran X'} \to \Ran X'
\]
is pseudo-proper and that $\matheur{A}'$ is factorizable, we see that for each non-empty finite set $I$, $\matheur{A}'|_{\oversetsupscript{X'}{\circ}{I}}$ lives in (perverse) cohomological degrees $\leq -3|I|$. 

Now, by Theorem~\ref{thm:Koszul_duality_connectivity_on_Ran}, we know that there exists an object
\[
	\mathfrak{a}' \in \Liestar(X')^{\leq c_{L}}
\]
such that
\[
	\matheur{A}' \simeq \Chev(\mathfrak{a'}).
\]

Theorem~\ref{thm:Chev_coChev_and_D_Ran} then implies that
\[
	D_{\Ran X'} \Chev(\mathfrak{a}') \simeq \coChev(D_{X'} \mathfrak{a}'),
\]
which is known to  be factorizable by Theorem~\ref{thm:factorizability_coChev}
\end{proof}

\begin{cor} \label{cor:B'_and_DRan_A'}
The map given in~\eqref{eq:B'_to_DA'} is an equivalence, i.e.
\[
	\matheur{B}' \simeq D_{\Ran X'} \matheur{A}' \teq\label{eq:B'_eq_DA'},
\]
and hence
\[
	\matheur{B} \simeq (j_{\Ran})_*\coChev D_{X'}\mathfrak{a}' \simeq \coChev j_* D_{X'} \mathfrak{a}'.
\]
\end{cor}
\begin{proof}
The first statement is a direct consequence of the proposition above and the equivalence~\eqref{eq:B'_eq_DA'_on_X'}, where as the second statement is the result of Proposition~\ref{prop:coChev_open_embedding}.
\end{proof}

\begin{prop} \label{prop:local_dual_implies_global_dual}
We have the following equivalence induced by Proposition~\ref{prop:DA'_factorizable}
\[
	C^*_c(\Ran X, \matheur{B}) \simeq C^*_c(\Ran X', \matheur{A}')^\vee.
\]
\end{prop}
\begin{proof}
We have the following equivalences
\begin{align*}
C^*_c(\Ran X, \matheur{B})
&\simeq C^*_c(\Ran X, \coChev j_* D_{X'} \mathfrak{a}') \teq\label{eq:use_cor:B'_and_DRan_A'} \\
&\simeq \coChev C^*_c(X, j_*D_{X'} \mathfrak{a}') \teq\label{eq:use_thm:coChev_and_C^*_c(Ran)} \\
&\simeq \coChev C^*(X, D_{X'} \mathfrak{a}')\\
&\simeq \coChev (C^*_c(X, \mathfrak{a}')^{\vee}) \\
&\simeq (\Chev(C^*_c(X', \mathfrak{a}')))^\vee 
\teq \label{eq:use_thm:Chev_coChev_and_D_Ran} \\
&\simeq C^*_c(\Ran X', \Chev\mathfrak{a}')^\vee \\
&\simeq C^*_c(\Ran X', \matheur{A}')^\vee.
\end{align*}
Here, \eqref{eq:use_cor:B'_and_DRan_A'}, \eqref{eq:use_thm:coChev_and_C^*_c(Ran)} and~\eqref{eq:use_thm:Chev_coChev_and_D_Ran} are due to Corollary~\ref{cor:B'_and_DRan_A'}, Theorem~\ref{thm:coChev_and_C^*_c(Ran)} and Theorem~\ref{thm:Chev_coChev_and_D_Ran} (applied to a point) respectively.
\end{proof}

\subsubsection{} Finally, as a corollary, we have the Atiyah-Bott formula. Indeed, we have
\[
	C_*^\red(\Bun_G)^\vee \simeq C^*_c(\Ran X', \matheur{A}')^\vee \simeq C^*_c(\Ran X, \matheur{B}) \simeq \Sym^{>0} C^*(X', M)
\]
where the first, second and third equivalences are due to Proposition~\ref{prop:non-abelian-poincare_sheaf_A}, Proposition~\ref{prop:local_dual_implies_global_dual}, and~\eqref{eq:explicit_C^*_c(Ran X, B)} respectively.

\appendix
\section{The $\addFil$ trick} \label{sec:appendix:addFil_trick}
In this appendix, we will quickly recall, without proof, a useful construction taken from~\cite[\S IV.2]{gaitsgory_study_2017}, which allows us to reduce many statements about $\matheur{P}$-algebras to trivial $\matheur{P}$-algebras, where $\matheur{P}$ is an operad in $\Vect$. Throughout this subsection, all categories without any further description will be assumed to be presentable, symmetric monoidal stable infinity over a field $k$ of characteristic 0. Moreover, functors between these categories are assumed to be continuous. 

All such categories, along with continuous functors between them, form a category, which we will use
\[
	\DGCatprescont^{\SymMon},
\]
to denote, or for simplicity
\[
	\DGCat^{\SymMon}.
\]

\subsection{Notations} For a symmetric monoidal category $\matheur{C}$, we denote the category of filtered objects in $\matheur{C}$
\[
	\matheur{C}^\Fil = \Fun(\mathbb{Z}, \matheur{C}),
\]
the category of functors from $\mathbb{Z}$ to $\matheur{C}$. Here, $\mathbb{Z}$ is a ordered set, viewed as a category. Similarly, we denote the category of graded objects
\[
	\matheur{C}^\gr = \Fun(\mathbb{Z}^{\set}, \matheur{C}),
\]
where $\mathbb{Z}^\set$ is a the discrete category, whose underlying underlying objects are the integers.\footnote{In~\cite{gaitsgory_study_2017}, it's called $\mathbb{Z}^{\mathrm{Spc}}$.}

\subsection{Functors} Now, we will recall several familiar functors between $\matheur{C}$, $\matheur{C}^\Fil$, and $\matheur{C}^\gr$.

\subsubsection{} Let
\[
	V = \cdots \to V_{n-1} \to V_n \to V_{n+1} \to \cdots,
\]
be an object in $\matheur{C}^\Fil$. Then, we define
\[
	\assgr: \matheur{C}^\Fil \to \matheur{C}^\gr
\]
to be the functor of taking the associated graded object
\[
	\assgr(V)_n = \coFib(V_{n-1} \to V_n),
\]
and
\[
	\oblv_{\Fil}: \matheur{C}^\Fil \to \matheur{C}
\]
to be the left Kan extension along
\[
	\mathbb{Z} \to \pt.
\]
Namely
\[
	\oblv_\Fil(V) = \colim_{n\in \mathbb{Z}} V_n.
\]

\subsubsection{} We also use
\[
	(\gr\to\Fil): \matheur{C}^\gr \to \matheur{C}^\Fil
\]
and
\[
\bigoplus: \matheur{C}^\gr \to \matheur{C}
\]
to denote the functor obtained by taking the left Kan extension along
\[
	\mathbb{Z}^\set \to \mathbb{Z},
\]
and
\[
	\mathbb{Z}^\set \to \pt
\]
respectively.

\subsubsection{} Note that the categories $\matheur{C}^\Fil$ and $\matheur{C}^\gr$ are equipped with a natural symmetric monoidal structure coming from $\matheur{C}$, and moreover, the functors $\assgr$, $\oblv_\Fil$, $\gr\to \Fil$, and $\bigoplus$ are naturally symmetric monoidal.

\subsubsection{Adding a filtration} Let
\[
	\addFil: \matheur{C} \to \matheur{C}^\Fil
\]
be the functor defined as follows: for an object $V$ in $\matheur{C}$,
\[
	\addFil(V)_n = 
	\begin{cases}
		V, &\text{when } n \geq 1, \\
		0, &\text{otherwise.}
	\end{cases}
\]

It's easy to see that
\[
	\bigoplus \circ \assgr \circ \addFil \simeq \oblv_\Fil \circ \addFil \simeq \id_\matheur{C}.
\]

\subsection{Interactions with algebras over an operad} Let $\matheur{P}$ be an operad in $\Vect$. Then we have the following pair of functors
\[
	\addFil: \matheur{P}\Palg(\matheur{C}) \to \matheur{P}\Palg(\matheur{C}^{\Fil^{>0}}) \qquad \text{and} \qquad \oblv_\Fil: \matheur{P}\Palg(\matheur{C}^{\Fil^{>0}}) \to \matheur{P}\Palg(\matheur{C}).
\]

\subsubsection{} \label{sec:appendix:addFil_oblvFil_commutative} Let
\[
	F: \DGCat^{\SymMon} \to \Cat_\infty
\]
be a functor, where $\Cat_\infty$ is the $\infty$-category of all $\infty$-categories. Suppose we have a continuous natural transformation
\[
	\Phi: \matheur{P}\Palg(-) \to F(-),
\]
i.e. morphisms between two objects in
\[
	\Fun(\DGCat^{\SymMon}, \Cat_\infty).
\]
Then from what we've discussed above, we have the following commutative diagram
\[
\xymatrix{
	\matheur{P}\Palg(\matheur{C}) \ar[r]^>>>>>>>\Phi & F(\matheur{C}) \\
	\matheur{P}\Palg(\matheur{C}^\Fil) \ar[u]^{\oblv_\Fil} \ar[r]^>>>>>\Phi & F(\matheur{C}^\Fil) \ar[u]^{\oblv_\Fil}
}
\]
which, combined with the fact that
\[
	\oblv_\Fil \circ \addFil \simeq \id_\matheur{C},
\]
implies that the following diagram also commutes
\[
\xymatrix{
	\matheur{P}\Palg(\matheur{C}) \ar[d]_{\addFil} \ar[r]^>>>>>>>\Phi & F(\matheur{C}) \\
	\matheur{P}\Palg(\matheur{C}^\Fil) \ar[r]^>>>>>\Phi & F(\matheur{C}^\Fil) \ar[u]^{\oblv_\Fil}
}
\]

\subsubsection{} Further composing the diagram above with $\assgr$ and $\bigoplus$ gives us the following commutative diagram
\[
\xymatrix{
	\matheur{P}\Palg(\matheur{C}) \ar[r]^>>>>>>>\Phi \ar[d]_{\addFil} & F(\matheur{C}) \\
	\matheur{P}\Palg(\matheur{C}^{\Fil^{>0}}) \ar[d]_{\assgr} \ar[r]^>>>>>{\Phi^{\Fil}} & F(\matheur{C}^{\Fil^{>0}}) \ar[u] ^{\oblv_\Fil}\ar[d]_{\assgr} \\
	\matheur{P}\Palg(\matheur{C}^{\gr^{>0}}) \ar[d]_{\bigoplus} \ar[r]^>>>>>{\Phi^{\gr}} & F(\matheur{C}^{\gr^{>0}}) \ar[d]_{\bigoplus} \\
	\matheur{P}\Palg(\matheur{C}) \ar[r]^>>>>>>>\Phi & F(\matheur{C})
} \teq\label{eq:addFil_trick_main_diagram}
\]

We will refer to this as the \emph{fundamental commutative diagram of the $\addFil$ trick}.

\subsubsection{} Now, suppose there are two natural transformations
\[
	\Phi_1, \Phi_2: \matheur{P}\Palg(-) \to F(-)
\]
equipped with a morphism between them
\[
	\alpha: \Phi_1 \to \Phi_2.
\]
Or more concretely, we have a compatible family of morphisms in $F(\matheur{C})$
\[
	\Phi_1(c) \to \Phi_2(c)
\]
parametrized by pairs $(\matheur{C}, c)$ where $c\in \matheur{C}$ and $\matheur{C} \in \DGCat^{\SymMon}$, and we want to prove that $\alpha$ is an equivalence.

\subsubsection{} The top square of the commutative diagram above implies that it suffices to show that 
\[
	\Phi_1 ^\Fil \circ \addFil \to \Phi_2^\Fil \circ \addFil
\]
is an equivalence. But since $\assgr$ and $\bigoplus$ are conservative, it suffices to show that
\[
	\bigoplus \circ \assgr \circ \Phi_1^{\Fil} \circ \addFil \to \bigoplus \circ \assgr \circ \Phi_2^{\Fil}\circ \addFil
\]
is an equivalence, which, due to the commutativity of the diagrams, is equivalent to
\[
	\Phi_1 \circ \bigoplus \circ \assgr \circ \addFil \to \Phi_2 \circ \bigoplus \circ \assgr \circ \addFil
\]
being an equivalence.

\subsubsection{} The crucial observation of~\cite[Prop. IV.2.1.4.6]{gaitsgory_study_2017} is the following
\begin{prop}
The functor
\[
	\bigoplus \circ \assgr\circ\addFil: \matheur{P}\Palg(\matheur{C}) \to F(\matheur{C})
\]
is canonically equivalent to $\triv_{\matheur{P}} \circ \oblv_{\matheur{P}}$, i.e.
\[
	\matheur{P}\Palg(\matheur{C}) \overset{\oblv_{\matheur{P}}}{\longrightarrow} \matheur{C} \overset{\triv_{\matheur{P}}}{\longrightarrow} \matheur{P}\Palg(\matheur{C}).
\]
\end{prop}

\subsubsection{} This implies that it suffices to prove
that
\[
	\Phi_1(c) \to \Phi_2(c)
\]
is an equivalence only for the case where $c$ is a trivial algebra.

\subsection{A general principle} More generally, suppose we want to prove a property of $\Phi(c)$ for some $c\in \matheur{P}\Palg(\matheur{C})$. Moreover, suppose this property is preserved under under $\oblv_{\Fil}$, and is conservative under $\bigoplus$ and $\assgr$. Then, it suffices to prove the case where $c$ has a trivial algebra structure.

\section{Co-filtration and $\addCoFil$}
\label{sec:appendix:cofiltration_addCoFil}
In this appendix, we will collect various notions that are dual to the one in \S\ref{sec:appendix:addFil_trick}. These are used in the body of the paper to give a proof of the $\addCoFil$ trick in a special case.

\subsection{Notations} For a symmetric monoidal category $\matheur{C}$, we denote the category of co-filtered objects
\[
	\matheur{C}^{\coFil} = \Fun(\mathbb{Z}^\op, \matheur{C}).
\]

We will also use $\matheur{C}^{\coFil^{>0}}$ to denote the full-subcategory of $\matheur{C}^{\coFil}$ consisting of objects supported in positive degrees. Similarly for graded objects $\matheur{C}^{\gr}$ and $\matheur{C}^{\gr^{>0}}$.

\subsection{Functors} As in the case of filtration, there are several familiar functors between $\matheur{C}, \matheur{C}^{\coFil}$, and $\matheur{C}^{\gr}$.

\subsubsection{} Let
\[
	V = \cdots \to V_{n+1} \to V_n \to V_{n-1} \to \cdots,
\]
be an object in $\matheur{C}^{\coFil}$. Then we define
\[
	\assgr: \matheur{C}^{\coFil} \to \matheur{C}^{\gr}
\]
to be the functor of taking the associated graded object
\[
	\assgr(V)_n = \Fib(V_n \to V_{n-1}),
\]
and
\[
	\oblv_{\coFil}: \matheur{C}^{\coFil} \to \matheur{C}
\]
to be the right Kan extension along
\[
	\mathbb{Z}^\op \to \pt.
\]
Namely
\[
	\oblv_{\coFil}(V) = \lim_{n\in \mathbb{Z}^{\op}} V_n.
\]

\subsubsection{} Note that the category $\matheur{C}^{\coFil}$ naturally inherits the monoidal structure coming from $\matheur{C}$. Moreover, the functor $\assgr$ is monoidal.

\subsubsection{} We also use
\[
	\prod: \matheur{C}^{\gr} \to \matheur{C}
\]
to denote the right Kan extension along
\[
	\mathbb{Z}^{\set} \to \pt.
\]
Namely
\[
	\prod((V_n)_{n\in \mathbb{Z}}) = \prod_{n\in \mathbb{Z}} V_n.	
\]

\subsubsection{Adding a co-filtration} We will use
\[
	\addCoFil: \matheur{C} \to \matheur{C}^{\coFil}
\]
to denote a functor defined as follows: for an object $V$ in $\matheur{C}$,
\[
	\addCoFil(V)_n =
	\begin{cases}
		V, & \text{when } n\geq 1, \\
		0, & \text{otherwise}.
	\end{cases}
\]

\subsection{Acknowledgments} The author would like to express his gratitude to D. Gaitsgory, without whose tireless guidance and encouragement in pursuing this problem, this work would not have been possible. The author is grateful to his advisor B.C. Ng\^o for many years of patient guidance and support.

This paper is revised while the author is a postdoc in Hausel group at IST Austria. We thank him and the group for providing a wonderful research environment. The author also gratefully acknowledges the support of the Lise Meitner fellowship ``Algebro-Geometric Applications of Factorization Homology'' No. M2751 of Austrian Science Fund (FWF).

\bibliography{connectivity_chiral_koszul}
\end{document}